\documentclass[10pt]{amsart}
\usepackage{fullpage,url,amssymb,mathrsfs,color}
\usepackage[all]{xy} 
\usepackage[OT2,T1]{fontenc}
\usepackage{mathpazo}
\usepackage{lmodern}

\numberwithin{equation}{section}

\usepackage{color}

\newcommand{\defi}[1]{\textsf{#1}} 

\newenvironment{romanenum}{\hfill \begin{enumerate} }{\end{enumerate}}
\newenvironment{alphenum}{\hfill \begin{enumerate} }{\end{enumerate}}

\DeclareSymbolFont{cyrletters}{OT2}{wncyr}{m}{n}
\DeclareMathSymbol{\Sha}{\mathalpha}{cyrletters}{"58}
\newcommand{\Aff}{{\mathbb A}}

\newcommand{\F}{{\mathbb F}}  \newcommand{\FF}{{\mathbb F}}

\newcommand{\PP}{{\mathbb P}}
 \newcommand{\QQ}{{\mathbb Q}}
\newcommand{\RR}{{\mathbb R}}
\newcommand{\ZZ}{{\mathbb Z}}

\def\bbar#1{\setbox0=\hbox{$#1$}\dimen0=.2\ht0 \kern\dimen0 \overline{\kern-\dimen0 #1}}
\newcommand{\Qbar}{{\overline{\mathbb Q}}}

\newcommand{\Fbar}{\overline{\F}} \newcommand{\FFbar}{\overline{\FF}}

\newcommand{\pp}{{\mathfrak p}}

\newcommand{\calB}{{\mathcal B}}

\newcommand{\calD}{{\mathcal D}}

\newcommand{\calF}{{\mathcal F}}

\newcommand{\calH}{{\mathcal H}}
\newcommand{\calI}{{\mathcal I}}

\newcommand{\calP}{{\mathcal P}}

\newcommand{\calS}{{\mathcal S}}

\newcommand{\calV}{{\mathcal V}}

\newcommand{\calX}{{\mathcal X}}

\newcommand{\OO}{{\mathcal O}}

\DeclareMathOperator{\ad}{ad}
\DeclareMathOperator{\scc}{sc}

\DeclareMathOperator{\Tr}{Tr}

\DeclareMathOperator{\Frob}{Frob}

\DeclareMathOperator{\Hom}{Hom}

\DeclareMathOperator{\Gal}{Gal}

\DeclareMathOperator{\ord}{ord}

\DeclareMathOperator{\Spec}{Spec}

\newcommand{\et}{{\operatorname{\'et}}}

\newcommand{\Ev}{{\operatorname{Ev}}}

\newcommand{\GL}{\operatorname{GL}}

\newcommand{\PSL}{\operatorname{PSL}}

\newcommand{\SO}{\operatorname{SO}}

\newcommand{\Sp}{\operatorname{Sp}}
\newcommand{\GSp}{\operatorname{GSp}}

\newcommand{\Or}{\operatorname{O}}
\newcommand{\Cent}{\operatorname{Cent}}

\DeclareMathOperator{\spin}{sp}
\DeclareMathOperator{\disc}{disc}

\newcommand{\ang}[2]{\langle #1,#2\rangle}

\def\Zz{\mathscr Z}

\def\CC{\mathbb C}

\newtheorem{theorem}{Theorem}[section]
\newtheorem{lemma}[theorem]{Lemma}
\newtheorem{corollary}[theorem]{Corollary}
\newtheorem{proposition}[theorem]{Proposition}

\theoremstyle{definition}

\newtheorem{example}[theorem]{Example}

\theoremstyle{remark}
\newtheorem{remark}[theorem]{Remark}

\definecolor{webcolor}{rgb}{0,0,1}
\definecolor{webbrown}{rgb}{.6,0,0}
\usepackage[
        colorlinks,
        linkcolor=webbrown,  filecolor=webbrown,  citecolor=webbrown, 
        backref,
        pdfauthor={David Zywina}, 
        pdftitle={Galois groups arising from families with big orthogonal monodromy},
]{hyperref}
\usepackage[alphabetic,backrefs,lite]{amsrefs} 

\begin{document}

\title[Galois groups arising from families with big orthogonal monodromy]
{Galois groups arising from families with big orthogonal monodromy}

\subjclass[2010]{Primary 11F80; Secondary 11G05, 14D05} 
\keywords{Galois representations; big monodromy, elliptic curves over function fields}
\author{David Zywina}
\address{Department of Mathematics, Cornell University, Ithaca, NY 14853, USA}
\email{zywina@math.cornell.edu}
\urladdr{http://www.math.cornell.edu/\~{}zywina}

\begin{abstract}

We study the Galois groups of polynomials arising from a compatible family of representations with big orthogonal  monodromy.  We show that the Galois groups are usually as large as possible given the constraints imposed on them by a functional equation and discriminant considerations.   As an application, we consider the Frobenius polynomials arising from the middle \'etale cohomology of hypersurfaces in $\PP_{\FF_q}^{2n+1}$ of degree at least $3$.  We also consider the $L$-functions of quadratic twists of fixed degree of an elliptic curve over a function field $\FF_q(t)$.  To determine the typical Galois group in the elliptic curve setting requires using some known cases of the Birch and Swinnerton-Dyer conjecture.  This extends and generalizes work of Chavdarov, Katz and Jouve.
\end{abstract}

\maketitle

\section{Introduction} 

\subsection{Some groups} \label{SS:W def}
We first define some of the groups that will arise as Galois groups in our applications.

For each integer $n\geq 1$, let $W_{2n}$ be the subgroup of signed permutations in $\GL_{n}(\ZZ)$, i.e., permutation matrices whose non-zero entries are allowed to be $\pm 1$.     Let $W_{2n}^+$ be the subgroup of $W_{2n}$ consisting of those elements that act evenly on the set $\{\pm e_1,\ldots, \pm e_n \}$.   The group $W_{2n}$ has order $2^n n!$ and is isomorphic to the Weyl group of the root systems $B_n$ and $C_n$.   The group $W_{2n}^+$ has order $2^{n-1} n!$ and is isomorphic to the Weyl group of the root system $D_n$. 

Now consider a polynomial $P\in \QQ[T]$ of degree $N>2$ that satisfies
\begin{align} \label{E:FE0}
T^N  P(1/T) = \varepsilon P(T)
\end{align}
for some $\varepsilon\in \{\pm 1\}$. Setting $T=1$ and $T=-1$ in the above equation, we find that $P(1)=0$ if $\varepsilon\neq 1$ and $P(-1)=0$ if $\varepsilon  \neq (-1)^N$.   
So by removing these obvious linear factors from $P$,  we obtain a polynomial
\begin{align} \label{E:initial f def}
f(T) := \begin{cases}
       P(T)/(1+\varepsilon T) & \text{ if $N$ is odd}, \\
       P(T)/(1-T^2) & \text{ if $N$ is even and $\varepsilon=-1$}, \\       
       P(T) & \text{ if $N$ is even and $\varepsilon=1$}
	\end{cases}
\end{align}
with rational coefficients and even degree $2n\geq 2$.   From (\ref{E:FE0}), we deduce that the polynomial $f$ is reciprocal, i.e., it satisfies $T^{2n}f(1/T)=f(T)$.

Let $\Gal(P)$ be the Galois group of a splitting field of $P$, equivalently of $f$, over $\QQ$.  Since $f$ is reciprocal, its distinct roots in $\Qbar-\{\pm 1\}$ are of the form $\alpha_1, \ldots,\alpha_m, \alpha_1^{-1},\ldots, \alpha_m^{-1}$ for an integer $0\leq m \leq n$.    Let $\iota\colon \{\alpha_1^{\pm 1}, \ldots, \alpha_m^{\pm 1}\} \to \{\pm e_1,\ldots, \pm e_m\}$ be the bijection satisfying $\iota(\alpha_i)=e_i$ and $\iota(\alpha_i^{-1})=-e_i$.   There is a unique injective homomorphism $\psi\colon \Gal(P) \hookrightarrow W_{2m}$ satisfying $\iota(\sigma(\alpha)) = \psi(\sigma) \cdot \iota(\alpha)$ for each root $\alpha\in \Qbar -\{\pm 1\}$ of $P$ and $\sigma\in \Gal(P)$.    So $\Gal(P)$ is isomorphic to a subgroup of $W_{2m}$ and hence also a subgroup of $W_{2n}$.   So $\Gal(P)$ is isomorphic to a subgroup of $W_{N-1}$ if $N$ is odd, $W_{N-2}$ if $N$ is even and $\varepsilon=-1$, and $W_N$ if $N$ is even and $\varepsilon=1$.  

Suppose that $P$ is separable, $N$ is even and $\varepsilon=1$.   If the discriminant of $P$ is a square, then $\Gal(P)$ will be isomorphic to a subgroup of $W_{N}^+$. 

\subsection{Example: smooth hypersurfaces over finite fields} 
\label{SS:hypersurface}

Fix an even integer $n\geq 2$ and an integer $d\geq 3$ with $(n,d)\neq (2,3)$.    Fix a finite field $\FF_q$ with cardinality $q$.  We define $U(\FF_q)$ to be the set of homogeneous polynomials in $ \FF_q[x_0,\ldots, x_{n+1}]$ of degree $d$, up to scalar multiplication by $\FF_q^\times$, that define a \emph{smooth} hypersurface in $\PP^{n+1}_{\FF_q}$.  

Take any $f\in U(\FF_q)$.   The \defi{zeta function} of the hypersurface $H_f$ in $\PP^{n+1}_{\FF_q}$ defined by $f$ is 
the power series 
\[
Z_f(T) = \exp\big(\sum_{n=1}^\infty |H_f(\FF_{q^n})| \cdot T^n/n \big).
\] 
One can show that $Z_f(T)$ is in a rational function in $T$ and moreover that
\[
Z_f(T)=1/\big(P_f(q^n T) \cdot {\prod}_{i=0}^{2n}(1-q^i T)\big)
\]
for a unique polynomial $P_f(T) \in \QQ[T]$ of degree $N :=(d-1)((d-1)^{n+1}+1)/d$.  Note that the integer $N$ depends only on $n$ and $d$.   The functional equation for $Z_f(T)$ implies that we have the relation $T^{N} P_f(1/T) = \varepsilon_f \,P_f(T)$ for a unique  $\varepsilon_f \in \{\pm 1\}$.   

We will describe the Galois group $\Gal(P_f)$ for a ``random'' $f\in U(\FF_q)$.   From \S\ref{SS:W def}, and using that $N$ is even if and only if $d$ is odd, we find that $\Gal(P_f)$ is isomorphic to a subgroup of $W_{N-1}$ if $d$ is even, $W_{N-2}$ if $d$ is odd and $\varepsilon_f=-1$, and $W_N$ if $d$ is odd and $\varepsilon_f=1$.

There is an additional constraint on the Galois group of $P_f$.  Suppose that $d$ is odd, $\varepsilon_f=1$ and $P_f$ is separable.   We will show later that the discriminant of $P_f$ is in $(-1)^{(d-1)/2} d \cdot (\QQ^\times)^2$.  Since $d$ is odd, we deduce that the discriminant of $P_f$ is a square if and only if $d$ is a square.  So if $d$ is odd, $\Gal(P_f)$ will be isomorphic to a subgroup of $W_N^+$.

The following theorem says that the Galois group of $P_f$ is as large as possible, given the above constraints, for a ``random'' polynomial $f\in U(\FF_q)$.

\begin{theorem} \label{T:hypersurfaces}  
For each prime power $q>1$, let $\delta(q)$ be the proportion of $f \in U(\FF_q)$ for which we have an isomorphism
\[
\Gal(P_f)\cong  \begin{cases}
       W_{N-1} & \text{ if $d$ is even},\\
       W_{N-2} & \text{ if $d$ is odd and $\varepsilon_f =-1$}, \\       
       W_N^+ & \text{ if $d$ is odd, $\varepsilon_f=1$, and $d$ is a square}, \\
       W_N & \text{ if $d$ is odd, $\varepsilon_f=1$, and $d$ is not a square.} \\       
	\end{cases}
\]
Then $\delta(q) \to 1$ as $q\to \infty$.
\end{theorem}

We will prove Theorem~\ref{T:hypersurfaces}  in \S\ref{S:hypersurfaces} by showing that it satisfies the general framework of Theorem~\ref{T:main 1}.  We will use some Hodge theory to compute the field $K$ of \S\ref{SS:the field K} which is needed to distinguish the cases when $d$ is odd and $\varepsilon_f=1$.

\begin{remark}
Let us briefly mention the excluded case where $n\geq 2$ is \emph{odd}.  Take any $d\geq 3$ and define $U(\FF_q)$ as before.   For any $f\in U(\FF_q)$, the zeta function of the hypersurface defined by $f$ will now be of the form $P_f(T)/\prod_{i=0}^{2n} (1-q^iT)$ for a unique polynomial $P_f(T)\in \QQ[T]$ of degree $N:=(d-1)((d-1)^{n+1}-1)/d$.

The description of the Galois group of $P_f$ for a ``random'' $f\in U(\FF_q)$ is now much more straightforward.  We have $\delta(q)\to 1$ as $q\to \infty$, where $\delta(q)$ is the proportion of $f\in U(\FF_q)$ for which $\Gal(P_f)$ is isomorphic to $W_N$.    This can be proved with the techniques of this paper and using the computations of Chavdarov (the work of Chavdarov will be described in \S\ref{SS:related results}).

For both even and odd $n$, the polynomial $P_f$ can be obtained from the characteristic polynomial of the $q$-th power Frobenius automorphism acting on the middle \'etale cohomology group $V:=H^n_{\et}((H_f)_{\FFbar_q}, \QQ_\ell)$ for a prime $\ell\nmid q$.   The important difference between the two cases is that the cup product $V\times V \to H^{2n}((H_f)_{\FFbar_q}, \QQ_\ell)\cong \QQ_\ell(-n)$
is symmetric when $n$ is even and skew-symmetric when $n$ is odd.   
\end{remark}

\begin{remark} \label{R:sign of hyper}
The sign $\varepsilon_f$ can be $+1$ or $-1$ and both occur with essentially equal likelihood.  More precisely, for a fixed $\varepsilon \in \{\pm 1\}$, we have $|\{f\in U(\FF_q): \varepsilon_f=\varepsilon\}|/|U(\FF_q)| \to 1/2$ as $q\to \infty$.
\end{remark}

\subsection{General setup} \label{SS:new setup}

Let $R$ be either a finite field or the ring of $S$-units in a number field $F$ with $S$ a finite set of non-zero prime ideals of $\OO_F$.
Let $U$ be a smooth scheme over $R$ of relative dimension at least $1$ with geometrically connected fibers.   

\subsubsection{Representations} \label{SSS:rep}
Fix a set of rational primes $\Sigma$ with Dirichlet density $1$ such that each $\ell\in \Sigma$ is not equal to the characteristic of $R$ and satisfies $\ell\geq 5$.   For each prime $\ell\in \Sigma$, we fix a continuous representation
\[
\rho_\ell\colon \pi_1(U_{R[1/\ell]}) \to \Or(M_\ell),
\]
where $M_\ell$ is an orthogonal space\footnote{The definitions of orthogonal spaces and orthogonal groups are recalled in \S\ref{SS:orthogonal spaces}.} over $\ZZ_\ell$.   Here, and throughout this article, $\pi_1$ will always refer to the \'etale fundamental group.   We will suppress the base point in our fundamental groups and hence its elements and representations will only be determined up to conjugacy.  Equivalent to giving $\rho_\ell$ is to give a lisse $\ZZ_\ell$-sheaf $\calH_\ell$ on $U_{R[1/\ell]}$ of free $\ZZ_\ell$-modules of finite rank with a symmetric autoduality pairing $\calH_\ell\times \calH_\ell\to \ZZ_\ell$.

 From $M_\ell$, we obtain orthogonal spaces $V_\ell:=M_\ell/\ell M_\ell$ and $\calV_\ell:= M_\ell \otimes_{\ZZ_\ell} \QQ_\ell$ over $\FF_\ell$ and $\QQ_\ell$, respectively.  Let
\[
\bbar\rho_\ell\colon \pi_1(U_{R[1/\ell]}) \to \Or(V_\ell)
\]
be the representation obtained by composing $\rho_\ell$ with the obvious reduction map.

\subsubsection{Compatibility} \label{SSS:compatibility}
Take any $R$-algebra $k$ that is a finite field and take any point $u\in U(k)$.    Let $\bbar{k}$ be a fixed algebraic closure of $k$.  For a prime $\ell \in \Sigma$ that is invertible in $k$, we have 
$u\in U(k)=U_{R[1/\ell]}(k)$.   Viewing $u$ as a morphism $\Spec k \to U_{R[1/\ell]}$, we obtain a group homomorphism $\Gal(\bbar{k}/k)=\pi_1(\Spec k) \to \pi_1(U_{R[1/\ell]})$ and we denote by $\Frob_u$ the image of the Frobenius automorphism of the extension $\bbar{k}/k$.  Observe that $\Frob_u$ lies in a well-defined conjugacy class of $\pi_1(U_{R[1/\ell]})$.     In particular, the polynomial
\[
P_u(T):=\det(I- \rho_\ell(\Frob_u) T)
\]
is well-defined and has coefficients in $\ZZ_\ell$.    

We shall further assume that the family of representations $\{\rho_\ell\}_{\ell\in\Sigma}$ are \defi{compatible}, i.e., the above polynomial $P_u(T)$ lies in $\QQ[T]$ and does not depend on the choice of $\ell$.    From our compatibility assumption, the rank of $M_\ell$ as a $\ZZ_\ell$-module does not depend on $\ell$; denote this common rank by $N$.   We shall assume that $N>2$.

 Since $\rho_\ell(\Frob_u)$ lies in $\Or(M_\ell)$, we have
\begin{equation} \label{E:root number}
T^N P_u(1/T) = \varepsilon_u P_u(T),
\end{equation}
where $\varepsilon_u:=\det(-\rho_\ell(\Frob_u)) \in \{\pm 1\}$.  From our compatibility assumption, the sign $\varepsilon_u$  does not depend on the choice of $\ell$.

\subsubsection{Big monodromy} \label{SSS:bm}
For each prime $\ell\in \Sigma$, let $\Or_{\calV_\ell}$ be the orthogonal group of $\calV_\ell$ as an algebraic group over $\QQ_\ell$.  For each field $k$ that is an $R[1/\ell]$-algebra, we can view $\rho_\ell(\pi_1(U_{\overline{k}}))$ as a subgroup of $\Or_{\calV_\ell}(\QQ_\ell)=\Or(\calV_\ell)$, where $\bbar{k}$ is a fixed algebraic closure of $k$.

We now make an additional ``\defi{big monodromy}'' assumption.  Assume that one of the following hold:
\begin{alphenum}
\item \label{bm-a}
The ring $R$ has characteristic $0$ and for any finite field $k$ that is a an $R$-algebra, the Zariski closure of $\rho_\ell(\pi_1(U_{\bbar{k}}))$ in $\Or_{\calV_\ell}$ is either $\SO_{\calV_\ell}$ or $\Or_{\calV_\ell}$ for a set of primes $\ell \in \Sigma$ with  Dirichlet density $1$. 
\item  \label{bm-b}
There is a subset $\Lambda \subseteq \Sigma$ with Dirichlet density $1$ such that for any finite field $k$ that is an $R$-algebra, we have
\[
\bbar\rho_\ell(\pi_1(U_{\bbar{k}})) \supseteq \Omega(V_\ell)
\]
for all primes $\ell \in \Lambda$ that are not equal to the characteristic of $k$, where $\Omega(V_\ell)$ is the commutator subgroup of $\Or(V_\ell)$.
\end{alphenum}
In fact, condition (\ref{bm-a}) implies condition (\ref{bm-b}), see~Corollary~\ref{C:equivalent monodromy}.

\subsubsection{The field $K$} \label{SS:the field K}
Suppose that $N$ is even.  We shall prove in \S\ref{S:the field K} that there is a unique extension $K/\QQ$ of degree at most $2$ such that for all sufficiently large $\ell \in \Sigma$, the prime $\ell$ splits in $K$ if and only if the orthogonal space $V_\ell$ is split (in the sense of \S\ref{SS:orthogonal}).  

Consider any point $u\in U(k)$ with $k$ a finite field that is an $R$-algebra.  Let $\Delta_u$ be the discriminant of $P_u$. If $\varepsilon_u=1$ and $P_u$ is separable, then $K=\QQ(\sqrt{\Delta_u})$, cf.~ Proposition~\ref{P:K criterion}.

\subsection{Main result} \label{SS:intro maximal}

Fix notation and assumptions as in \S\ref{SS:new setup}.    Take any $u\in U(k)$, where $k$ is an $R$-algebra that is a finite field.   Let $\Gal(P_u)$ be the Galois group of a splitting field of $P_u$ over $\QQ$.     From \S\ref{SS:W def} and (\ref{E:root number}), we find that $\Gal(P)$ is isomorphic to a subgroup of $W_{N-1}$ if $N$ is odd, $W_{N-2}$ if $N$ is even and $\varepsilon_u=-1$, and $W_N$ if $N$ is even and $\varepsilon_u=1$. 

Suppose that $N$ is even, $\varepsilon_u=1$ and $P_u$ is separable.   If $K=\QQ$, then the discriminant $\Delta_u$ of $P_u$ is a square and hence $\Gal(P_u)$ is isomorphic to a subgroup of $W_N^+$.\\ 

The following theorem describes the Galois group $\Gal(P_u)$ for a ``random'' $u\in U(k)$.  The group $\Gal(P_u)$ is usually as large as possible given the constraints discussed above.

\begin{theorem} \label{T:main 1}
For a finite field $k$ that is an $R$-algebra, we define $\delta(k)$ to be the proportion of $u\in U(k)$ for which we have 
\begin{align} \label{E:Galois specific}
\Gal(P_u) \cong \begin{cases}
      W_{N-1} & \text{ if $N$ is odd}, \\
      W_{N-2} & \text{ if $N$ is even and $\varepsilon_u=-1$}, \\
      W_{N} & \text{ if $N$ is even, $\varepsilon_u=1$ and $K\neq \QQ$}, \\
      W_{N}^+ & \text{ if $N$ is even, $\varepsilon_u=1$ and $K=\QQ$}
\end{cases}
\end{align}
(and set $\delta(k)=0$ when $U(k)$ is empty).
Then
\[
\lim_{k,\,|k|\to \infty} \delta(k) = 1,
\]
where the limit is over finite fields $k$ that are $R$-algebras with increasing cardinality.
\end{theorem}

\begin{remark}
\begin{romanenum}
\item
Theorem~\ref{T:main 1} answers a question of Katz on what the maximal Galois groups are, see the end of \S1 of \cite{katzreport} where it is asked in the setting of elliptic curves (which we will discuss in \S\ref{SS:EC setup}).   Katz's guess is the same as (\ref{E:Galois specific}) except he predicts that $W_N^+$ is the group for $N$ even and $\varepsilon_u=1$.
\item Jouve proved a special case of Theorem~\ref{T:main 1}, in the context of elliptic curves, where he showed that $\Gal(P_u)$ is either equal to $W_{2n}$ or $W_{2n}^+$ for the appropriate $n$.  See Remark~\ref{R:Jouve} for the precise result.
\end{romanenum}
\end{remark}

\subsection{An effective version} \label{SS:effective version}

Fix notation and assumptions as in \S\ref{SS:new setup}.   Now assume that $R=\FF_q$ is a finite field with odd cardinality and that $U$ is a smooth affine curve over $\FF_q$ that is geometrically integral.  Let  $C/\FF_q$ be the smooth projective curve that contains $U$ as a Zariski open subvariety.   Let $g$ be the genus of $C$ and set $b$ be the number of points in the set $C(\FFbar_q)- U(\FFbar_q)$.  

We also assume that condition (\ref{bm-b}) of \S\ref{SS:new setup} holds with a set of primes $\Lambda$ having \emph{natural} density $1$.  Finally, we assume that the representations $\{\bbar\rho_\ell\}_{\ell\in \Sigma}$ are all tamely ramified. 

In this special setting, the following gives an effective version of Theorem~\ref{T:main 1}.

\begin{theorem} \label{T:main 2}
For all $n\geq 1$, we have
\[
\delta(\FF_{q^n})= 1 + O\big(  2^{2g+b} (2g+b)\,  q^{-n/(N^2-N+6)} \log (q^n)  \big),
\]
where the implicit constant depends only on $\Sigma$.
\end{theorem}

In particular, note that $1-\delta(\FF_{q^n})$ decays exponentially as a function of $n\geq 1$.  This strengthens a result of Jouve that we will recall in Remark~\ref{R:Jouve}.

\subsection{Example: $L$-functions of twists of an elliptic curve} \label{SS:EC setup}

Fix an elliptic curve $E$ defined over the function field $\FF_q(t)$, where $q$ is a power of a prime $p\geq 5$.     Assume that $E$ has multiplicative reduction at some place $v\neq \infty$ of $\FF_q(t)$, where $\infty$ is the place of $\FF_q(t)$ with uniformizer $t^{-1}$.   Let $m(t)$ be the monic squarefree polynomial in $\FF_q[t]$ whose irreducible factors correspond to the places $v\neq \infty$ for which $E/\FF_q(t)$ has bad reduction.  

Fix an integer $d\geq 1$.   For each integer $n\geq 1$, define the set
\[
U_d(\FF_{q^n}) = \big\{u\in\FF_{q^n}[t]: u \text{ squarefree, }\deg(u)=d,\; \gcd(u,m)=1\big\};
\]
it will serve as a parameter space for quadratic twists of $E$.  Identifying a polynomial with the tuple of its coefficients, we can view $U_d(\F_{q^n})$ as the $\FF_{q^n}$-points of an open subvariety $U_d$ of $\Aff^{d+1}_{\FF_q}$.    

Take any $u\in U_d(\FF_{q^n})$ and let $E_u$ be an elliptic curve over $\FF_{q^n}(t)$ obtained by taking a \defi{quadratic twist} of $E$ by $u$.   Let $v$ be a place of $\FF_{q^n}(t)$ and let $\FF_v$ be the corresponding residue field.  When $E_u$ has good reduction at $v$, we define the integer $a_v=q^{\deg v} +1 -|E(\FF_v)|$, where $E(\FF_v)$ is the $\FF_v$-points of a good model of $E$ over the local ring at $v$ and $\deg v$ is the degree of the field extension $\FF_v/\FF_q$.    If $E$ has bad reduction at $v$, define $a_v=1$, $-1$ or $0$ if $E$ has split multiplicative, non-split multiplicative or additive reduction, respectively, at $v$.  The \defi{$L$-function} of the elliptic curve $E_u$ over $\FF_{q^n}(t)$ is the power series
 \[
 L(T,E_u) := \prod_{v \text{ good}} (1-a_vT^{\deg v} + q^{\deg v} T^{2\deg v})^{-1} \cdot \prod_{v \text{ bad}} (1-a_vT^{\deg v})^{-1},
 \]
 where the product is over the places $v$ of $\FF_{q^n}(t)$.    Moreover, one can show that $L(T,E_u)$ is a polynomial. 
 Define the polynomial
 \[
 P_u(T):= L(T/q^n, E_u) \in \QQ[T].
 \]
The degree $N_d$ of $P_u(T)$ depends only on $E$ and $d$; we will give an explicit formula below.   
 The \defi{functional equation} of $L(T,E_u)$ says that
 \[
 T^{N_d} P_u(1/T) = \varepsilon_u P_u(T)
 \]
for a unique $\varepsilon_u \in \{\pm 1\}$ called the \defi{root number} of $E_u$.   
\\
 
For each place $v$ of $\FF_q(t)$, we can assign a \defi{Kodaira symbol} to the elliptic curve $E$ after base extending to the local field $\FF_q(t)_v$; the symbol can be computed quickly using Tate's algorithm.     For each place $v$ of $\FF_q(t)$, we define integers $f_v(E)$, $\gamma_v(E)$ and $b_v(E)$ using the following table.
\begin{center}
\begin{tabular}{c||c|c|c|c|c|c|c|c|c|c}
Kodaira symbol at $v$ & $\text{I}_0$ &$\text{I}_n \,(n\geq 1) $& $\text{II}$ & $\text{III}$ &  $\text{IV}$ & $\text{I}_0^*$ &  $\text{I}_n^*$ ($n\geq1$) &  $\text{IV}^*$ & $\text{III}^*$  & $\text{II}^*$ \\\hline 
$f_v$ & $0$ & $1$ & $2$ & $2$ & $2$ & $2$  & $2$   & $2$ & $2$ & $2$ \\
$\gamma_v$ & $1$ & $n/\gcd(2,n)$ & $1$ & $1$ & $3$ & $1$ & $2/\gcd(2,n)$ & $3$ & $1$  & $1$\\
$b_v$ & $0$ & $0$& $1$ & $1$ & $1$ & $0$ & $1$ & $1$ & $1$ & $1$ \\
\end{tabular}
\end{center}
The common degree of the polynomials $P_u(T)$ is 
\begin{align*}
N_d &= f_\infty(E_{t^d}) + \sum_{v\neq \infty} f_v(E)\deg v - 4+ 2d,
\end{align*}
where the sum is over the places $v\neq \infty$ of $\FF_q(t)$ and $E_{t^d}/\FF_q(t)$ is the quadratic twist of $E$ by $t^d$.   We also define the integers
\[
D_d := \gamma_\infty(E_{t^d})\cdot \prod_{v \neq  \infty} \gamma_v(E)^{\deg v}\quad  \text{ and }\quad B := \sum_{v\neq \infty} b_v(E)\deg v.
\]

The following describes the Galois group of the $L$-function of $E_u/\FF_{q^n}(t)$ when $E$ is twisted by a ``random'' $u\in U_d(\FF_{q^n})$.

\begin{theorem} \label{T:main EC}
Fix an integer $d\geq 1$ so that $N_d \geq \max\{6B,3\}$.   Assume further that $d\geq 2$ or that there is a place $v\neq \infty$ of $\FF_q(t)$ for which $E$ has Kodaira symbol $\operatorname{I}_0^*$.
For each $n\geq 1$, let $\delta(q^n)$ be the proportion of $u\in U_d(\FF_{q^n})$ for which we have an isomorphism
\begin{align} \label{E:EC ratio}
\Gal(P_u)=\Gal(L(T,E_u)) \cong \begin{cases}
      W_{N_d-1} & \text{ if $N_d$ is odd}, \\
      W_{N_d-2} & \text{ if $N_d$ is even and $\varepsilon_u=-1$}, \\
      W_{N_d} & \text{ if $N_d$ is even, $\varepsilon_u=1$, and $(-1)^{N_d/2} D_d$ is not a square}, \\
      W_{N_d}^+ & \text{ if $N_d$ is even, $\varepsilon_u=1$, and $(-1)^{N_d/2} D_d$ is a square}
\end{cases}
\end{align}
(and set $\delta(q^n)=0$ when $U_d(\FF_{q^n})$ is empty).  Then $\delta(q^n) \to 1$ as $n\to\infty$.
\end{theorem}

\begin{remark}
\label{R:ECs}
\begin{romanenum}
\item 
Note that the conditions on $d$ in Theorem~\ref{T:main EC} hold for all sufficiently large $d$; our constraint on $d$ is used to apply a big monodromy theorem of Hall.  
\item
Using the work of Katz and Hall, we will verify that the polynomials $P_u$ arise from representations as in the axiomatic setup of \S\ref{SS:new setup}. The remaining task is to compute the associated field $K$ from \S\ref{SS:the field K} when $N_d$ is even; this is needed to distinguish the two possible cases when $\varepsilon_u=1$.
\item
Suppose that $N_d$ is even and take any polynomial $u\in U_d(\FF_{q^n})$ for which $\varepsilon_u=1$ and $P_u$ is separable.   Denote the discriminant of $P_u$ by $\Delta_u$.   One can show that the square class $\Delta_u\cdot (\QQ^\times)^2$ is independent of the choice of $u$.   Distinguishing the last two cases of (\ref{E:EC ratio}) is a result of this square class being  $(-1)^{N_d/2} D_d \cdot (\QQ^\times)^2$.

How does one prove this?  Using that $P_u$ is reciprocal and separable, one can prove that
\begin{align*}
\Delta_u\cdot (\QQ^\times)^2 &= (-1)^{N_d/2} P_u(1) P_u(-1)(\QQ^\times)^2\\
&= (-1)^{N_d/2} L(1/q^n, E_u)\,  L(-1/q^n, E_u) \cdot(\QQ^\times)^2\\
&= (-1)^{N_d/2} L(1/q^n, E_u)\,  L(1/q^n, E_{\alpha u}) \cdot(\QQ^\times)^2,
\end{align*}
where $\alpha \in \FF_{q^n}^\times$ is any choice of non-square. 
Since the values $L(1/q^n, E_u)$ and $L(1/q^n, E_{\alpha u})$ are non-zero, the \defi{Birch and Swinnerton-Dyer conjecture} (BSD) give an explicit expression for them in terms of interesting invariants of $E_u$ and $E_{\alpha u}$, respectively.  This part of BSD for elliptic curves over global function fields has been proved by Tate and Milne.   Several of the invariants that arise, like the cardinality of the (finite!) Tate--Shafarevich group, are squares and hence do not need to be computed.   Proving that $ (-1)^{N_d/2}L(1/q^n, E_u) L(1/q^n, E_{\alpha u}) \in(-1)^{N_d/2}D_d \cdot (\QQ^\times)^2$ is then essentially an application of the Tate algorithm.   For details and background, see \S2.4 of \cite{orthogonal}.  The paper \cite{orthogonal}, which proves the Inverse Galois Problem for several groups of the form $\Omega(V_\ell)$, were motivated by these computations.
\item
If $N_d$ is even, then the integer $(-1)^{N_d/2}D_d$ depends only on the parity of $d$.
\item \label{R:ECs eps}
Both possibilities for $\varepsilon_u$ occur.  Moreover, we have $|\{u\in U_d(\FF_{q^n}): \varepsilon_u=\varepsilon\}|/|U(\FF_{q^n})| \to 1/2$ as $n\to \infty$ for each $\varepsilon \in \{\pm 1\}$.
\end{romanenum}
\end{remark}

\begin{example}
As an example consider a prime $q=p\geq 5$ and let $E/\FF_p(t)$ be the elliptic curve defined by $y^2=x(x-1)(x-t)$.  Fix an integer $d\geq 2$.  

The only places $v\neq \infty$ of $\FF_p(t)$ for which $E$ has bad reduction are those with uniformizers $t$ and $t-1$, and the Kodaira symbol is $\operatorname{I}_2$ at both places.   The elliptic curve $E_{t^d}$  has bad reduction at $\infty$ and the Kodaira symbol is $\operatorname{I}_2$ when $d$ is odd and $\operatorname{I}_2^*$ when $d$ is even.  We thus have $N_d=2d-1$ if $d$ is odd and $N_d=2d$ if $d$ is even.   We have $B=0$ and $d\geq 2$, so the conditions of Theorem~\ref{T:main EC} hold.   When $N_d$ is even, equivalently $d$ is even, we have $D_d=1$ and hence $(-1)^{N_d/2} D_d = (-1)^d=1$.

The set $U_d(\FF_{p^n})$ consists of all separable degree $d$ polynomials $u\in \FF_{p^n}[t]$ with $u(0)u(1)\neq 0$.    If $d$ is even and $u\in U_d(\FF_{p^n})$, one can show that $\varepsilon_u=1$ if and only if $u(0)u(1)$ is a square in $\FF_{p^n}^\times$;  we can express $\varepsilon_u$ as a product of a local root numbers that are easy to compute, cf.~Theorem~3.1 of \cite{MR2183392}.   For each $n\geq 1$, let $\delta(p^n)$ be the proportion of $u\in U_d(\FF_{p^n})$ for which we have an isomorphism
\begin{align*} 
\Gal(P_u)\cong \begin{cases}
      W_{2d-2} & \text{ if $d$ is odd or if $d$ is even and $u(0)u(1)$ is not a square in $\FF_{p^n}^\times$}, \\
      W_{2d}^+ & \text{ if $d$ is even and $u(0)u(1)$ is a square in $\FF_{p^n}^\times$.} \\
\end{cases}
\end{align*}
Theorem~\ref{T:main EC} in this case says that $\delta(p^n)\to 1$ as $n\to \infty$.
\end{example}

We now give an explicit version where we restrict to certain $1$ dimension subvarieties of $U_d$.

\begin{theorem}  \label{T:main 2 EC}  
Fix an integer $d\geq 1$ as in Theorem~\ref{T:main EC} and fix a polynomial $g(t)\in U_{d-1}(\FF_q)$.   Let $\delta(q^n)$ be the proportion of $c \in \FF_{q^n}$ for which the polynomial $u:=(t-c)g(t) \in \FF_{q^n}[t]$ is squarefree and relatively prime to $m(t)$, and for which the Galois group $\Gal(P_u)=\Gal(L(T,E_u))$ satisfies (\ref{E:EC ratio}).
Then 
\[
\delta(q^n)=1+O\big(2^{\deg m+d}(\deg m + d)\, q^{-n/(N_d^2-N_d+6)} \log (q^n)\big),
\]
where the implicit constant depends only on the $j$-invariant of $E$.
\end{theorem}

\begin{remark} \label{R:Jouve}
Theorem~\ref{T:main 2 EC} is a strengthening of the main result of Jouve, cf.~Theorem~4.3 of \cite{MR2539184}.  Jouve bounds the number of $c\in \FF_q$ with $m(c)g(c)\neq 0$ such that $\Gal\!\big(L(T,E_{(t-c)g(t)}/\FF_{q}(t))\big)$ does not equal the appropriate Galois group $W_{2n}$ or its subgroup $W_{2n}^+$.   Jouve obtains a bound of the form 
\[
O\big(N_d^2\, |G|\, q^{1-1/(3.5N_d^2-3.5N_d+2)}\log q\big)  
\]
for $d$ sufficiently large, where the implicit constant depends only on the $j$-invariant of $E$ and $G$ is a certain finite group.    A bound for $|G|$ is not given in \cite{MR2539184} but one can show that $|G|\leq 2^{\deg m +d}$ using the approach of Lemma~\ref{L:finite abelianization}.
\end{remark}

We will prove Theorems~\ref{T:main EC} and \ref{T:main 2 EC} in \S\ref{S:cohomological interpretation} by applying the axiomatic setup of \S\ref{SS:new setup} and \S\ref{SS:effective version}.

\subsection{Some related results} \label{SS:related results}

This paper was motivated by the work of Chavdarov for which we now recall a special case.     Let $U$ be a geometrically irreducible variety over $\FF_q$ of positive dimension.   Consider a compatible family of continuous representations $\{\rho_\ell\}_{\ell}$ with $\rho_\ell\colon \pi_1(U)\to \GSp_{2g}(\ZZ_\ell)$.  For each $u\in U(\FF_{q^n})$, let $P_u \in \QQ[T]$ be the corresponding polynomial of degree $2g$ arising from the representations $\rho_\ell$.   We also make a \emph{big monodromy} assumption: suppose that the image of $\rho_\ell(\pi_1(U_{\FFbar_q}))$ modulo $\ell$ is $\Sp_{2g}(\FF_\ell)$ for all sufficiently large $\ell$. 

 Let $\delta(q^n)$ be the proportion of $u\in U(\FF_{q^n})$ for which the Galois group of $P_u$ is isomorphic to $W_{2g}$.   Theorem~2.1 of \cite{MR1440067} then says that $\delta(q^n)\to 1$ as $n\to \infty$.  Note the description of $\Gal(P_u)$ for a ``random'' $u$ is much simpler than that of Theorem~\ref{T:main 1}.  One key reason is that the algebraic groups $\GSp_{2g}$ and $\Sp_{2g}$ that arise in Chavdarov's work are connected, while orthogonal groups are not connected.  Also the group $\Sp_{2g}$ is simply connected, while special orthogonal groups are not.\\
 
 Katz has proved a theorem similar to Theorem~\ref{T:main 1}, in the setting of $L$-functions of elliptic curves, except showing that $P_u$ with the obvious linear factors removed is irreducible, cf.~Theorem~4.1 of \cite{katzreport}.

As noted in Remark~\ref{R:Jouve}, Jouve proved an analogue of Theorem~\ref{T:main 2}, in the setting of $L$-functions of elliptic curves,  showing that the Galois group of $P_u$ for a ``random'' $u$ is isomorphic to either $W_{2n}^+$ or $W_{2n}$ for an appropriate $n$.  One of the main motivations of this paper is to distinguish between these two cases.

\subsection{Overview}

We now give a brief overview.  In \S\ref{S:orthogonal}, we describe some basic facts about orthogonal spaces and groups.   In particular in \S\ref{SS:counting elements with sep polynomial}, we study the cardinality of certain conjugacy classes of orthogonal groups over finite fields.  When $N$ is even, the field $K$ from \S\ref{SS:the field K} will be discussed in \S\ref{S:the field K}.  

Fix notation and assumptions as in \S\ref{SS:new setup}.   Consider a polynomial $P_u \in \QQ[T]$.   At the beginning of \S\ref{SS:intro maximal}, we have given some constraint on the group $\Gal(P_u)$.   How do we show that $\Gal(P_u)$ satisfies (\ref{E:Galois specific}), i.e., is as large as possible?   The idea is fundamental to Galois theory; we will consider the reduction of $P_u$ modulo various primes $\ell$ and compute how it factors in $\FF_\ell[T]$.  If we see enough different kinds of factorizations, we will be able to prove that $\Gal(P_u)$ is as large as possible.  The following proposition, which we will prove in \S\ref{S:proof of criterion for sieving}, is a key ingredient in the proof of our main theorems.

\begin{proposition} 
\label{P:criterion for sieving}
  For each $\ell\in \Sigma$, there are subsets $C_1(V_\ell),\ldots, C_6(V_\ell)$ of $\Or(V_\ell)$ such that the following hold:
\begin{romanenum}
\item  \label{P:criterion for sieving a}
$C_i(V_\ell)$ is stable under conjugation by $\Or(V_\ell)$.
\item  \label{P:criterion for sieving b}
There are positive absolute constants $c_1$ and $c_2$ such that if $\ell \in \Sigma$ satisfies $\ell\geq c_1$, then
\[
\frac{|C_i(V_\ell)\cap \kappa|}{|\kappa|} \geq \frac{c_2}{N^2}
\]
for all cosets $\kappa$ of $\Omega(V_\ell)$ in $\Or(V_\ell)$ and all integers $1\leq i\leq 6$.
\item   \label{P:criterion for sieving c}
Take any $u\in U(k)$, where $k$ is a finite field that is an $R$-algebra.   Suppose that for each $1 \leq i\leq 6$ there is a prime $\ell \in \Sigma$, not equal to the characteristic of $k$, such that $\bbar\rho_{\ell}(\Frob_u) \subseteq C_i(V_\ell)$.  Then the Galois group of $P_u(T)$ satisfies (\ref{E:Galois specific}).
\end{romanenum}
\end{proposition} 

The representations $\{\rho_\ell\}$ are not independent, i.e., a condition imposed on $P_u$ modulo one prime can restrict the possible reductions modulo other primes.  In \S\ref{S:big monodromy}, we use our big monodromy assumption, and some group theory, to show that the image of the representation $\prod_{\ell\in D} \bbar\rho_\ell$ is large for all finite subsets $D\subseteq \Lambda$, where $\Lambda$ is an appropriate subset of $\Sigma$ with Dirichlet density $1$.   This controls how dependent the representations $\rho_\ell$ are.
   
Theorem~\ref{T:main 1} and Theorem~\ref{T:main 2} will be proved in \S\ref{S:proof of main 1} and \S\ref{S:proof of main 2}, respectively.   Our examples from \S\ref{SS:hypersurface} and \S\ref{SS:EC setup}, will be proved in \S\ref{S:hypersurfaces} and \S\ref{S:cohomological interpretation}, respectively.   In Appendix~\ref{S:The Selberg sieve}, we state a general version of Selberg's sieve.  For convenience, we state some equidistribution bounds in Appendix~\ref{S:B}.

\section{Orthogonal groups and characteristic polynomials} \label{S:orthogonal}

\subsection{Orthogonal spaces} \label{SS:orthogonal spaces}

Let $R$ be an integral domain whose characteristic is not $2$.   An \defi{orthogonal space} $M$ over $R$ is a free $R$-module $M$ of finite rank equipped with a symmetric $R$-bilinear pairing $\ang{\:}{\:}\colon M\times M \to R$ which induces an isomorphism $M\to \Hom_R(M,R),$ $m\mapsto \ang{m}{\cdot}$.    

A \defi{homomorphism} of orthogonal spaces is an $R$-module homomorphism that is compatible with the respective pairings.   The \defi{orthogonal group} of $M$, denoted by $\Or(M)$,  is the group of automorphisms of the orthogonal space $M$.   Let $\SO(M)$ be the kernel of the determinant map $\det\colon \Or(M)\to \{\pm 1\}$.   

\subsection{Finite fields} \label{SS:orthogonal}

Fix a finite field $\FF$ with odd characteristic.   Let $V$ be an orthogonal space over $\FF$ of dimension $N\geq 1$.   The \defi{discriminant} of $V$, denoted by $\disc(V)$, is the coset in $\FF^\times/(\FF^\times)^2$ represented by $\det(\ang{v_i}{v_j})$, where $v_1,\ldots, v_N$ is any basis of $V$ over $\FF$.   Up to isomorphism, there are two orthogonal spaces of dimension $N$ over $\FF$; these orthogonal spaces are distinguishable by their discriminants.  If $N$ is even, we say that $V$ is \defi{split} if $\disc(V)=(-1)^{N/2}(\FF^\times)^2$ and \defi{non-split} otherwise.

For each $v\in V$ with $\ang{v}{v}\neq 0$, we have a reflection $r_v \in \Or(V)$ defined by $x\mapsto x-2\ang{x}{v}/\ang{v}{v} \cdot v$.  Let 
\[
\spin_V\colon \Or(V)\to \FF^\times/(\FF^\times)^2
\] 
be the \defi{spinor norm}.  The spinor norm is a homomorphism that can be characterized by the property that it is satisfies $\spin(r_v)= \ang{v}{v}\cdot (\FF^\times)^2$ for all $v\in V$ with $\ang{v}{v}\neq 0$.    We will denote $\spin_V$ by $\spin$ if $V$ is clear from context.

\begin{lemma} \label{L:disc comes from spin}
We have $\disc(V)=\spin_V(-I)$.
\end{lemma}
\begin{proof}
Let $v_1,\ldots, v_N$ be an orthogonal basis of $V$.  We have $-I = r_{v_1} r_{v_2} \cdots r_{v_N}$, so
\[
\disc(V) = \det(\ang{v_i}{v_j}) \cdot (\FF^\times)^2= {\prod}_i \ang{v_i}{v_i} \cdot (\FF^\times)^2= {\prod}_i \spin(r_{v_i}) = \spin(-I).  \qedhere
\]
\end{proof}

Define $\Omega(V)$ to be the simultaneous kernels of the homomorphisms $\det\colon \Or(V)\to\{\pm 1\}$ and $\spin\colon \Or(V)\to \FF^\times/(\FF^\times)^2$.  
The following lemma recalls some basic facts about these groups; see \cite{Atlas}*{\S2.4} for a good exposition of the groups $\Omega(V)$; proofs can be found in \S3.7 and \S3.11 of \cite{MR2562037} for $N\neq 5$ and $N=4$, respectively.

\begin{lemma} \label{L:Omega facts}    Suppose that $N\geq 3$ and $q>3$.  Let $Z$ be the center of $\Omega(V)$.
\begin{romanenum}
\item 
The map $\det\times\spin\colon \Or(V)/\Omega(V) \to \{\pm 1\}\times \FF^\times/(\FF^\times)^2$ is an isomorphism.  
\item 
The group $Z$ is either $\{I\}$ or $\{\pm I\}$. 
\item 
The group $\Omega(V)/Z$ is simple except when $N=4$ and $V$ is split.
\item 
If $N=4$ and $V$ is split, then $\Omega(V)/Z \cong \PSL_2(\FF)\times \PSL_2(\FF)$.
\item \label{L:Omega facts v}  
The group $\Omega(V)$ is perfect.   In particular, $\Omega(V)$ is the commutator subgroup of $\Or(V)$.
\end{romanenum}
\end{lemma}

\begin{remark} \label{R:alternate Omega}
We now give an alternate description of the group $\Omega(V)$.  Let $\SO_V$ be the obvious algebraic group over $\FF$; it is semisimple and has a simply connected cover $\pi\colon G\to \SO_V$.  The group $\Omega(V)$ is equal to $\pi(G(\FF))$.
\end{remark}

\begin{remark}
We gave an alternate definition of $\Omega(V_\ell)$ in condition (\ref{bm-b}) in \S\ref{SSS:bm}.  Since $\ell\geq 5$ and $N>2$, these definitions agree by  Lemma~\ref{L:Omega facts}(\ref{L:Omega facts v}).
\end{remark}

The follows lemma allows to compute the spinor norm for some elements in $\Or(V)$ directly from their characteristic polynomials.

\begin{lemma} \label{L:disc and spin rules} 
Take any $A\in \Or(V)$ and set $P(T)=\det(I-AT)$.
\begin{romanenum}
\item \label{L:disc and spin rules i} 
If $P(-1) \neq 0$, then $\spin(A) =2^N P(-1) (\FF^\times)^2$.
\item \label{L:disc and spin rules ii} 
If $P(1) \neq 0$, then $\spin(A) =2^N P(1) \disc(V)$.
\item \label{L:disc and spin rules iii} 
If $P(1) \neq 0$ and $P(-1) \neq 0$, then $\disc(V) =P(1)P(-1) (\FF^\times)^2$.
\end{romanenum}
\end{lemma}
\begin{proof}
If $P(-1)\neq 0$, then Zassenhaus \cite{MR0148760}*{p.446} shows that  $\spin(A)$ equals  
\[
\det((I+A)/2) (\FF^\times)^2 = 2^N \det(I+A)(\FF^\times)^2 = 2^N P(-1) (\FF^\times)^2.
\]  
This gives (\ref{L:disc and spin rules i}), and part (\ref{L:disc and spin rules ii}) follows by applying (\ref{L:disc and spin rules i}) with the matrix $-A$ and using Lemma~\ref{L:disc comes from spin}.   Finally, (\ref{L:disc and spin rules iii}) follows directly from (\ref{L:disc and spin rules i}) and (\ref{L:disc and spin rules ii}).
\end{proof}

Take any $A\in \Or(V)$ and define $P(T)=\det(I-AT)$.  We have
\begin{equation}  \label{L:FE A}
T^N P(1/T) = \det(-A) P(T) = (-1)^N \det(A) P(T).
\end{equation}
Substituting $1$ and $-1$ into (\ref{L:FE A}), we have $P(1)=(-1)^N \det(A) P(1)$ and $P(-1)=\det(A) P(-1)$.  
So $P$ is divisibly by $1-T$ if $\det(A)= (-1)^{N+1}$ and by $1+T$ if $\det(A)=-1$.    Removing these obvious linear factors from $P$,  we have the polynomial
\[
f(T) := \begin{cases}
       P(T) & \text{ if $N$ is even and $\det(A)=1$}, \\
       P(T)/(1-T^2) & \text{ if $N$ is even and $\det(A)=-1$}, \\       
       P(T)/(1-\det(A)T) & \text{ if $N$ is odd}. \\
	\end{cases}
\]
Using (\ref{L:FE A}), we find that $f(T)\in \FF[T]$ is reciprocal, i.e., $T^{\deg f} f(1/T) = f(T)$.   The polynomial $f(T)$ is monic and has even degree.

\subsection{Reciprocal polynomials}

\begin{lemma} \label{L:basic separability}
Fix a field $K$ whose characteristic is not $2$.   Let $f\in K[T]$ be a monic reciprocal polynomial of even degree $2n\geq 2$.
\begin{romanenum}
\item \label{L:basic separability 0}
We have $f(T)=T^n h(T+1/T)$ for a unique polynomial $h \in K[T]$. The polynomial $h$ is monic of degree $n$.
\item \label{L:basic separability i}
We have 
\begin{align} \label{E: discf vs disch}
\disc(f)=(-1)^n f(1)f(-1) \disc(h)^2 = h(2)h(-2) \disc(h)^2.   
\end{align}
In particular, $f$ is separable if and only if $h$ is separable and $h(2)h(-2) \neq 0$.
\item \label{L:basic separability ii}
Suppose that $K=\FF$ is a finite field.    Further suppose that $h$ is irreducible and $h(2)h(-2) \neq 0$. 

\begin{itemize}
\item
If $h(2)h(-2)$ is not a square in $\FF$, then $f$ is irreducible of degree $2n$ in $\FF[T]$.  
\item
If $h(2)h(-2)$ is a square in $\FF$, then $f$ is the product of two irreducible polynomials of degree $n$ in $\FF[T]$.  
\end{itemize}

\end{romanenum}
\end{lemma}
\begin{proof}
See \cite{MR2381481}*{Lemma 6} for the existence in part (\ref{L:basic separability 0}); the uniqueness is clear.  Denote the discriminant of $f$ and $h$ by $\disc(f)$ and $\disc(h)$, respectively.  It is straightforward to show that 
\begin{equation} \label{E:basic separability}
\disc(f)=(-1)^n f(1)f(-1) \disc(h)^2 = h(2)h(-2) \disc(h)^2,
\end{equation}
see \cite{MR2381481}*{\S3} for example.   Part (\ref{L:basic separability i}) is now immediate from (\ref{E:basic separability})

Now suppose that $h \in \FF[T]$ is irreducible and satisfies $h(\pm 2)\neq 0$.  From part (\ref{L:basic separability i}), $f$ is separable.   Let $\alpha \in \FFbar$ be any root of $f$; we have $\alpha\neq 0$ since $f$ is reciprocal.   The extension $\FF(\alpha+\alpha^{-1})/\FF$ has degree $n$ since $\alpha+\alpha^{-1}$ is a root of $h$ and $h$ is irreducible of degree $n$.   The extension $\FF(\alpha)/\FF$ thus has degree $n$ or $2n$.   Since $\alpha$ was an arbitrary root of $f$, we find that $f$ is either irreducible of degree $2n$ or the product of two irreducible polynomials of degree $n$.    From Theorem 7 in \cite{MR2381481}, we deduce that $f$ is irreducible if and only if $(-1)^n f(1)f(-1)=h(2)h(-2)$ is not a square in $\FF$.
\end{proof}

\subsection{Counting elements with a given separable characteristic polynomial} 
\label{SS:counting elements with sep polynomial}

Fix an orthogonal space $V$ of dimension $N\geq 2$ over a finite field $\FF$ with odd cardinality $q$. 

In this section, we give an explicit formula for the number of $A\in \Or(V)$ for which $\det(I-AT)$ is equal to a fixed \emph{separable} polynomial in $\FF[T]$.   These computations are of independent interest.\\
 
 Fix an integer $n\geq 1$ and a monic, separable and reciprocal polynomial $f\in \FF[T]$ of degree $2n$.   There is a unique (monic) polynomial $h\in \FF[T]$ of degree $n$ such that $f(T)=T^n h(T+1/T)$.   From Lemma~\ref{L:basic separability}(\ref{L:basic separability i}) and the assumption that $f$ is separable, we find that $h$ is separable and $h(2) h(-2)= (-1)^n f(1) f(-1)$ is non-zero.  

Let $h_1,\ldots, h_r \in \FF[T]$ be the monic irreducible factors of $h$.   Define $e_i=1$ if $h_i(2)h_i(-2) \in \FF$ is a square, otherwise set $e_i=-1$.

\begin{proposition} \label{P:conjugacy for N even}
Let $V$ be an orthogonal space of even dimension $N=2n$ over $\FF$.   Let $C$ be the set of $A\in \Or(V)$ for which $\det(I-AT)=f(T)$.  
\begin{romanenum}
\item \label{P:conjugacy for N even i}
If $\disc(V)\neq f(1)f(-1)(\FF^\times)^2$, then $C=\emptyset$.
\item \label{P:conjugacy for N even ii}
If $\disc(V)=f(1)f(-1)(\FF^\times)^2$, then $C$ is a conjugacy class of $\Or(V)$ and 
\[
|C|/|\!\Or(V)| = q^{-n} \prod_{i=1}^r (1-e_i/q^{\deg h_i})^{-1}. 
\]
We have $\det(A)=1$ and $\spin(A)= f(-1) (\FF^\times)^2$ for all $A\in C$.
\end{romanenum}
\end{proposition}
\begin{proof}
We use the background material in Appendix~A of \cite{MR1947324} which holds for a general field whose characteristic is not $2$.   Any $A\in \Or(V)$ with $\det(I-A)=f(T)$ has determinant $1$ since $f(T)$ is reciprocal.

Consider pairs $(V,A)$ consisting of an orthogonal space $V$ over $\FF$ with an automorphism $A\in \SO(V)$.  We say that two such pairs $(V,A)$ and $(V',A')$ are \emph{equivalent} if there is an isomorphism $B\colon V\to V'$ of orthogonal spaces for which $A'=B\circ A  \circ B^{-1}$.  Let $\calV(f)$ be the set of equivalence classes of pairs $(V,A)$ for which $\det(I-AT)=f(T)$.

 We have an extension of $\FF$-algebras $K/k$, where $K=\FF[x]/(f(x))$, $k=\FF[y]/(h(y))$ and $y=x+x^{-1}$.  Since $f(x)$ and $h(y)$ are separable, the algebras $K$ and $k$ will be products of finite extensions of $\FF$.  Let $\iota\colon K\to K$ be the automorphism which fixes $k$ and satisfies $\iota(x)=x^{-1}$.  Let $N_{K/k} \colon K \to k$ be the norm map $\alpha\mapsto \alpha\cdot \bbar\alpha$, where we set $\bbar\alpha=\iota(\alpha)$.

For each $\xi \in k^\times$, define the $\FF$-vector space $V_\xi:=K$ and endow it with the $\FF$-valued pairing $\langle\alpha,\beta\rangle_\xi = \Tr_{K/\FF}(\xi \alpha \overline{\beta})$.  With this bilinear form, $V_\xi$ is an orthogonal space of dimesion $2n$ over $\FF$.  The map $A_\xi\colon K\to K$ defined by $A_\xi(\alpha) = x\alpha$ is an automorphism of the orthogonal space $V_\xi$.  By construction, we have $\det(I-A_\xi T) = f(T)$.  If $\xi,\lambda \in k^\times$ satisfy $\xi\lambda^{-1}= N_{K/k}(\delta)$ for some $\delta\in K^\times$, then the map $B:K\to K$ defined by $B(\alpha)=\delta\alpha$ gives an equivalence between $(V_\xi,A_\xi)$ and $(V_\lambda,A_\lambda)$.  We thus have a well-defined map
\begin{align*}
\phi \colon k^\times/N_{K/k}(K^\times) &\to \calV(f),\quad \xi \mapsto (V_\xi,A_\xi). 
\end{align*}
Theorem~A.2 of \cite{MR1947324} shows that $\phi$ is a bijection.   

Since $\FF$ is finite and $K$ and $k$ are the product of finite extension fields of $\FF$, we know that $N_{K/k}\colon  K^\times\to k^\times$ is surjective and hence $|\calV(f)|=1$.   So there is a pair $(V,A)$, unique up to equivalence, that satisfies $\det(I-AT)=f(T)$.   In particular, the set $C$ of $B\in \SO(V)$ with $\det(I-BT)=f(T)$ is the conjugacy class of $A$ in $\Or(V)$.   By Lemma~\ref{L:disc and spin rules}(\ref{L:disc and spin rules iii}), we have $\disc(V)=f(1)f(-1) (\FF^\times)^2$ and hence the uniqueness of the equivalence class $(V,A)$ gives part (\ref{P:conjugacy for N even i}).   We have $\spin(A)=f(-1) (\FF^\times)^2$ by Lemma~\ref{L:disc and spin rules}(\ref{L:disc and spin rules i}).  It remains to compute $|C|/|\Or(V)|$.  

Since $C$ is the conjugacy class of $A$ in $\Or(V)$, we have
\[
|C|/|\Or(V)|=|\{ B \in \Or(V): \det(I-BT) =f(T) \}|/|\Or(V)|= 1/|\Cent_{\Or(V)}(A)|.
\]
Theorem~A.2 of \cite{MR1947324} also shows that $\Cent_{\Or(V)}(A)\cong \ker(N_{K/k}\colon K^\times\to k^\times)$.  For $1\leq i \leq r$, define $f_i(T):=T^{\deg h_i} h_i(T+1/T)$.  We thus have
\[
\Cent_{\Or(V)}(A) \cong {\prod}_{i=1}^r \ker( N_{K_i/k_i}\colon K_i^\times \to k_i^\times),
\]
where we have the extension of $\FF$-algebras $K_i/k_i$ with $K_i:=\FF[x]/(f_i(x))$ and $k_i:=\FF[y]/(h_i(y))$.   Since $N_{K/k}\colon K^\times\to k^\times$ is surjective, we have  $|\Cent_{\Or(V)}(A)| = \prod_{i=1}^r |K_i^\times|/|k_i^\times|$.

Suppose that $e_i=-1$, and hence $f_i(T)$ is irreducible by Lemma~\ref{L:basic separability}(\ref{L:basic separability ii}).  Then $K_i/k_i$ is a quadratic extension of finite fields, so $|K_i^\times|/|k_i^\times| = |k_i| + 1 = q^{\deg h_i} +1 = q^{\deg h_i} (1-e_i/q^{\deg h_i})$.   

Suppose that $e_i=1$, and hence $f_i(T)$ is the product of two irreducible polynomials of degree $\deg h_i$ by Lemma~\ref{L:basic separability}(\ref{L:basic separability ii}).  Then $K_i$ is isomorphic to the product of two fields isomorphic to $k_i$,  so $|K_i^\times|/|k_i^\times| = |k_i| - 1 = q^{\deg h_i} -1 = q^{\deg h_i} (1-e_i/q^{\deg h_i})$.  

Therefore, $ |C|/|\Or(V)|$ equals
 \[
1/|\Cent_{\Or(V)}(A)| = \big({\prod}_{i=1}^r q^{\deg h_i}(1-e_i/q^{\deg h_i})\big)^{-1}=q^{-n}{\prod}_{i=1}^r (1-e_i/q^{\deg h_i})^{-1}.    \qedhere
 \]
\end{proof}

\begin{proposition}  \label{P:conjugacy for N even 2}
Let $V$ be an orthogonal space of dimension $2n+2$ over $\FF$ and fix a coset $\beta \in \FF^\times/(\FF^\times)^2$.  Let $C_\beta$ be the set of $A\in \Or(V)$ for which $\det(I-AT)=(1-T^2)f(T)$ and $\spin(A)= \beta$.  Then $C_\beta$ is a conjugacy class of $\Or(V)$ and 
\[
|C_\beta|/|\!\Or(V)| = \frac{1}{4} q^{-n} \prod_{i=1}^r (1-e_i/q^{\deg h_i})^{-1}.
\]
\end{proposition}
\begin{proof}
Let $V_1$ be the orthogonal space of dimension $2n$ over $\FF$ with $\disc(V_1)=f(1)f(-1)(\FF^\times)^2$.   Let $V_2$ and $V_3$ be orthogonal spaces of dimension $1$ over $\FF$ such that $\disc(V_2)=f(-1)\beta$ and $\disc(V_3)= f(1)\beta \disc(V)$.  We have $\disc(V_1\oplus V_2\oplus V_3)= f(1)f(-1)\cdot f(-1)\beta\cdot f(1) \beta \disc(V) = \disc(V)$.   Therefore, the orthogonal spaces $V$ and $V_1\oplus V_2\oplus V_3$ are isomorphic; without loss of generality, assume that $V=V_1 \oplus V_2 \oplus V_3$.

By Proposition~\ref{P:conjugacy for N even}(\ref{P:conjugacy for N even ii}), there is an $A_1\in \SO(V_1)$ such that $\det(I-A_1T)=f(T)$ and $\spin(A_1)=f(-1) (\FF^\times)^2$.  Let $A\in \Or(V)$ be the automorphism that acts as $A_1$ on $V_1$, $-I$ on $V_2$, and $I$ on $V_3$.   Therefore, $\det(I-AT) = f(T)(1+T)(1-T) = f(T)(1-T^2)$.   We have 
\[
\spin(A)=\spin(A_1) \spin(-I_{V_2}) \spin(I_{V_3})=f(-1) \cdot \spin(-I_{V_2})\cdot 1 = f(-1)\disc(V_2) = \beta,
\]
where we have used Lemma~\ref{L:disc comes from spin}.    So $A$ belongs to $C_\beta$.

Now take any $B \in C_\beta$.  Let $W_2$ and $W_3$ be the (one-dimensional) eigenspaces of $B$ corresponding to the eigenvalues $-1$ and $1$, respectively.  Let $W_1$ be the subspace of $V$ perpendicular to $W_2$ and $W_3$.  With the pairing from $V$, the $W_i$ are orthogonal spaces and $V=W_1\oplus W_2\oplus W_3$.  The automorphism $B$ acts on $W_1$; denote by $B_1 \in \Or(W_1)$ the restriction of $B$ to $W_1$.   We have $\beta= \spin(B) = \spin(B_1)\spin(-I_{W_2}) \spin(I_{W_3})=\spin(B_1)\spin(-I_{W_2})=\spin(B_1) \disc(W_2)$.  By Lemma~\ref{L:disc and spin rules}, we have $\spin(B_1)=f(-1) (\FF^\times)^2$ and $\disc(W_1)=f(1)f(-1)(\FF^\times)^2$.    Therefore, $\disc(W_2)=f(-1)\beta$ and $\disc(W_3)=\disc(V)\disc(W_1)\disc(W_2)=f(1)\beta\disc(V)$.  

By comparing discriminants, we have isomorphisms $\varphi_1\colon V_1\xrightarrow{\sim} W_1$, $\varphi_2\colon V_2 \xrightarrow{\sim} W_2$ and $\varphi_3\colon V_3 \xrightarrow{\sim} W_3$ of orthogonal spaces.   By Proposition~\ref{P:conjugacy for N even}, we may take $\varphi_1$ so that $B_1= \varphi_1\circ A_1 \circ \varphi_1^{-1}$.   The automorphisms $\varphi_1,\varphi_2,\varphi_3$ give rise to an automorphism $\varphi\in \Or(V)$ such that $B = \varphi \circ A \circ \varphi^{-1}$.    Therefore, $C_\beta$ is a conjugacy class of $\Or(V)$.

Since $C_\beta$ is a conjugacy class of $\Or(V)$, it has cardinality $|\!\Or(V)|/|\Cent_{\Or(V)}(A)|$.   The above argument shows that $\Cent_{\Or(V)}(A)$ is equal to 
\[
\Cent_{\Or(V_1)}(A_1) \times \Cent_{\Or(V_2)}(-I) \times \Cent_{\Or(V_3)}(I)= \Cent_{\Or(V_1)}(A_1) \times \{\pm I\} \times \{\pm I\}.
\]
Therefore,
\[
|C_\beta|/|\!\Or(V)| = 1/|\Cent_{\Or(V)}(A)|= 1/|\Cent_{\Or(V_1)}(A_1)| \cdot 1/2 \cdot 1/2 = \tfrac{1}{4} q^{-n} \prod_{i=1}^r (1-e_i/q^{\deg h_i})^{-1},
\]
where the last equality uses Proposition~\ref{P:conjugacy for N even}.
\end{proof}

Finally, we consider orthogonal spaces of odd dimension.
 
 \begin{proposition} \label{P:conjugacy for N odd}
Let $V$ be an orthogonal space of dimension $2n+1$ over $\FF$.  Fix an $\varepsilon\in \{\pm 1\}$.    Let $C$ be the set of $A\in \Or(V)$ for which $\det(I-AT)=(1-\varepsilon T)f(T)$.  Then $C$ is a conjugacy class of $\Or(V)$ and 
\[
|C|/|\!\Or(V)| = \frac{1}{2} q^{-n} \prod_{i=1}^r (1-e_i/q^{\deg h_i})^{-1}.
\]
For $A\in C$, we have $\det(A)=\varepsilon$, $\spin(A)= f(-1) (\FF^\times)^2$ if $\varepsilon=1$ and $\spin(A)=f(1)\disc(V)$ if $\varepsilon=-1$.
\end{proposition}
\begin{proof}
Let $V_1$ and $V_2$ be the orthogonal spaces of dimension $2n$ and $1$, respectively, over $\FF$ with $\disc(V_1)=f(1)f(-1)(\FF^\times)^2$ and $\disc(V_2)=f(1)f(-1) \disc(V)$.   We have $\disc(V_1\oplus V_2) = \disc(V)$, so $V$ and $V_1\oplus V_2$ are isomorphic.  Without loss of generality, we may assume that $V=V_1 \oplus V_2$.

By Proposition~\ref{P:conjugacy for N even}, there is an $A_1\in \SO(V_1)$ such that $\det(I-A_1T)=f(T)$.  Let $A\in \Or(V)$ be the automorphism that acts as $A_1$ on $V_1$ and as scalar multiplication by $\varepsilon$ on $V_2$.  Therefore, $\det(I-AT) = f(T)(1-\varepsilon T)$ and hence $A\in C$.

Now take any $B \in C$.  Let $W_2$ be the (one-dimensional) eigenspace of $B$ corresponding to the eigenvalues $\varepsilon$.  Let $W_1$ be the subspace of $V$ perpendicular to $W_2$.  With the pairing from $V$, $W_1$ and $W_2$ are orthogonal spaces and $V=W_1\oplus W_2$.  The automorphism $B$ acts on $W_1$; denote by $B_1 \in \Or(W_1)$ the restriction of $B$ to $W_1$.   

By Proposition~\ref{P:conjugacy for N even}, we have $\disc(W_1)=f(1)f(-1)(\FF^\times)^2$, so $\disc(V_1)=\disc(W_1)$.  Therefore, $\disc(V_2) = \disc(V) \disc(V_1)$ equals $\disc(W_2)=\disc(V) \disc(W_1)$.  So there are isomorphisms $\varphi_1\colon V_1\xrightarrow{\sim} W_1$ and $\varphi_2\colon V_2 \xrightarrow{\sim} W_2$ of orthogonal spaces.   By Proposition~\ref{P:conjugacy for N even}, we may take $\varphi_1$ so that $B_1= \varphi_1\circ A_1 \circ \varphi_1^{-1}$.   The automorphisms $\varphi_1$ and $\varphi_2$ give rise to an automorphism $\varphi\in \Or(V)$ such that $B = \varphi \circ A \circ \varphi^{-1}$.    

Therefore, $C$ is a conjugacy class of $\Or(V)$ containing $A$ and hence has cardinality equal to $|\!\Or(V)|/|\Cent_{\Or(V)}(A)|$.   The above argument shows that $\Cent_{\Or(V)}(A)$ is equal to  
\[
\Cent_{\Or(V_1)}(A_1) \times \Cent_{\Or(V_2)}(\varepsilon I_{V_2}) = \Cent_{\Or(V_1)}(A_1) \times \{\pm I\}.
\]  
Therefore,
\[
|C|/|\!\Or(V)| = 1/|\Cent_{\Or(V_1)}(A_1)| \cdot 1/2 = \tfrac{1}{2} q^{-n} \prod_{i=1}^r (1-e_i/q^{\deg h_i})^{-1},
\]
where the last equality uses Proposition~\ref{P:conjugacy for N even}.

Finally, we compute $\spin(A)$.  We have $\spin(A)=\spin(A_1) \spin(\varepsilon I_{V_2})= f(-1) \spin(\varepsilon I_{V_2})$, where the last equality uses Proposition~\ref{P:conjugacy for N even}.  If $\varepsilon=1$, then $\spin(A)=f(-1)(\FF^\times)^2$.    We have $\spin(-I_{V_2})=\disc(V_2)=f(1)f(-1) \disc(V)$, so if $\varepsilon=-1$, then $\spin(A)= f(1)\disc(V)$.  
\end{proof}

\section{Big monodromy} \label{S:big monodromy}

Fix notation and assumptions as in \S\ref{SS:new setup}.  Let $F$ be the fraction field of $R$.   When $R$ has characteristic $0$, and hence $F$ is a number field, we have $R=\OO_F[S^{-1}]$ for a finite set $S$ of non-zero prime ideals of $\OO_F$.

For each finite subset $D$ of $\Sigma$, define the representation
\[
\bbar\rho_D:= \prod_{\ell \in D} \bbar\rho_\ell \colon \pi_1(U_{R[D^{-1}]}) \to \prod_{\ell\in D} \Or(V_\ell)
\]
and the subgroup $G_D^g:=\bbar{\rho}_D(\pi_1(U_{\bbar{F}}))$ of $\prod_{\ell\in D} \Or(V_\ell)$.   The goal of this section is to prove the following two propositions.

\begin{proposition} \label{P:big monodromy a}
There is a subset $\Lambda \subseteq \Sigma$ with Dirichlet density $1$ such that the inclusion
\begin{align} \label{E:big monodromy a}
\bbar\rho_D(\pi_1(U_{\bbar{k}})) \supseteq \prod_{\ell\in D}\Omega(V_\ell)
\end{align}
holds for all finite subsets $D\subseteq \Lambda$ and all finite fields $k$ that are $R$-algebras with characteristic not in $D$.
Moreover, $G_D^g \supseteq  \prod_{\ell\in D}\Omega(V_\ell)$. 

If $R$ is a finite field and condition (\ref{bm-b}) in \S\ref{SS:new setup} holds, then we may take $\Lambda$ to be the set of primes from condition  (\ref{bm-b}).
\end{proposition}

\begin{proposition} \label{P:big monodromy b}
Suppose that $R$ has characteristic $0$.   There is a finite set $S' \supseteq S$ of non-zero prime ideals of $\OO_F$ and a subset $\Lambda \subseteq \Sigma$ with Dirichlet density $1$ such that the inclusion
\begin{align} \label{E:big monodromy b}
\bbar\rho_D(\pi_1(U_{\bbar{k}}))=G_D^g
\end{align}
holds for all finite subsets $D\subseteq \Lambda$ and all finite fields $k$ that are $\OO_F[S'^{-1}]$-algebras with characteristic not in $D$.
\end{proposition}

\begin{remark}
Note that any subgroup of $\prod_{\ell\in D}\Or(V_\ell)$ containing $\prod_{\ell\in D}\Omega(V_\ell)$ is a normal subgroup.  This explains why (\ref{E:big monodromy a}) and (\ref{E:big monodromy b}) are well-defined without the fundamental groups have explicit base points.
\end{remark}

\begin{corollary} \label{C:equivalent monodromy}
Condition (\ref{bm-a}) of \S\ref{SS:new setup} implies condition (\ref{bm-b}).
\end{corollary}
\begin{proof}
We obtain condition (\ref{bm-b}) by taking singleton sets $D$ in Proposition~\ref{P:big monodromy a}.
\end{proof}

\subsection{Proof of Propositions \ref{P:big monodromy a} and \ref{P:big monodromy b}}

\begin{lemma} \label{L:big monodromy initial}
Fix a finite field $k$ that is an $R$-algebra.   There is a subset $\Lambda \subseteq \Sigma$ with Dirichlet density $1$ such that $\bbar\rho_\ell(\pi_1(U_{\bbar{k}})) \supseteq \Omega(V_\ell)$ holds for all primes $\ell \in \Lambda$ that are not equal to the characteristic of $k$.  If condition (\ref{bm-b}) in \S\ref{SS:new setup} holds, then we may take $\Lambda$ to be the set of primes from condition (\ref{bm-b}).
\end{lemma}
\begin{proof}
The lemma is immediate if condition (\ref{bm-b}) holds, so we may assume that condition (\ref{bm-a}) in \S\ref{SSS:bm} holds.  For $\ell\in \Sigma$ not equal to the characteristic of $k$, let
\[
\varrho_\ell\colon \pi_1(U_k)\to \Or_{\calV_\ell}(\QQ_\ell)
\]
be the representation obtained by specializing $\rho_\ell$.  By  condition (\ref{bm-a}), there is a subset $\Lambda\subseteq \Sigma$ with Dirichlet density $1$, that does not contain the characteristic of $k$, such that the neutral component of the Zariski closure of $\varrho_\ell(\pi_1(U_k))$ is $\SO_{\calV_\ell}$ for all $\ell\in \Lambda$.  

Now take any $\ell\in \Lambda$.  We have a connected and semisimple group scheme $H_\ell:=\SO_{M_\ell}$ over $\ZZ_\ell$ and base extension by $\QQ_\ell$ gives $\SO_{\calV_\ell}$.   Let $H_\ell^{\ad}$ be the quotient of $H_\ell$ by its center and let $H_\ell^{\scc}$ be the simply connected cover of $H_\ell$.   Denote by $\pi\colon H_\ell^{\scc}\to H_\ell$ and $\sigma \colon H_\ell \to H_\ell^{\ad}$ the natural homomorphisms.  Define 
\[
\Gamma_\ell:= \varrho_\ell(\pi_1(U_k)) \cap \SO_{\calV_\ell}(\QQ_\ell);
\] 
it is a compact subgroup of $H_\ell(\ZZ_\ell)=\SO_{M_\ell}(\ZZ_\ell)$.   Define the subgroup 
\[
\Gamma_\ell^{\scc}:=\{ g\in H_\ell^{\scc}(\QQ_\ell): \sigma(\pi(g)) \in \sigma(\Gamma_\ell)\}
\]
of $H_\ell^{\scc}(\QQ_\ell)$.  Observe that $\Gamma_\ell^{\scc} \subseteq H_\ell^{\scc}(\ZZ_\ell)$. 

We now apply a theorem of Larsen.  By Theorem 3.17 of \cite{MR1370110},  there is a subset $\Lambda'\subseteq \Lambda$ with  Dirichlet density $1$ such that $\Gamma_\ell^{\scc}$ is a maximal compact subgroup of $H_\ell^{\scc}(\QQ_\ell)$ for all $\ell\in \Lambda'$.   Therefore, $\Gamma_\ell^{\scc}=H_\ell^{\scc}(\ZZ_\ell)$ for all $\ell\in \Lambda'$ since $H_\ell^{\scc}(\ZZ_\ell)$ is a compact subgroups of $H_\ell^{\scc}(\QQ_\ell)$, respectively.  

Take any prime $\ell\in \Lambda'$ satisfying $\ell \geq 11$.  Let $\bbar{\Gamma}_\ell$ be the image of $\Gamma_\ell$ in $\SO(V_\ell)$.   The group $\pi(H_\ell^{\scc}(\FF_\ell))$  is equal to the commutator subgroup of $H_\ell(\FF_\ell)=\SO(V_\ell)$; for example, see \S1.2 of \cite{MR1370110} and note that a simply connected group is a product of simple simply connected groups.  Therefore, $\pi(H_\ell^{\scc}(\FF_\ell))=\Omega(V_\ell)$ by Lemma~\ref{L:Omega facts}(\ref{L:Omega facts v}).  So the group generated by $\bbar{\Gamma}_\ell$ and the center of $\SO(V_\ell)$ contains $\Omega(V_\ell)$.   The commutator subgroup of $\bbar{\Gamma}_\ell$ contains $\Omega(V_\ell)$ and hence $\bbar\rho_\ell(\pi_1(U_k)) \supseteq \Omega(V_\ell)$.  Since $\bbar{k}/k$ is an abelian extension and $\Omega(V_\ell)$ is perfect, we have $\bbar\rho_\ell(\pi_1(U_{\bbar{k}})) \supseteq \Omega(V_\ell)$.
\end{proof}

\begin{lemma} \label{L:big monodromy D}
Take any finite field $k$ that is an $R$-algebra.   Let $\Lambda$ be the set of primes from Lemma~\ref{L:big monodromy initial}.   Then for any finite subset $D\subseteq \Lambda$ not containing the characteristic of $k$, we have
\[
\bbar\rho_D(\pi_1(U_{\bbar{k}})) \supseteq \prod_{\ell\in D} \Omega(V_\ell).
\]
\end{lemma}
\begin{proof}
 Let $H$ be the commutator subgroup of $\bbar\rho_D(\pi_1(U_{\bbar{k}}))$; it is a subgroup of $\prod_{\ell \in D} \Omega(V_\ell)$ which is the commutator subgroup of $\prod_{\ell \in D} \Or(V_\ell)$ by Lemma~\ref{L:Omega facts}(\ref{L:Omega facts v}).    For each $\ell \in D$, we have $\bbar\rho_\ell(\pi_1(U_{\bbar{k}}))\supseteq \Omega(V_\ell)$ by our choice of $\Lambda$  and hence the commutator subgroup of $\bbar\rho_\ell(\pi_1(U_{\bbar{k}}))$ equals $\Omega(V_\ell)$ since $\Omega(V_\ell)$ is perfect by Lemma~\ref{L:Omega facts}(\ref{L:Omega facts v}).  Therefore, the projection homomorphism $H\to \Omega(V_\ell)$ is surjective for all $\ell\in D$.    
 
Fix $\ell\in D$.   Lemma~\ref{L:Omega facts} implies that the only non-abelian simple group in the composition series of $\Omega(V_\ell)$  is $\Omega(V_\ell)/Z_\ell$ where $Z_\ell$ is the center, \emph{except} when $N=4$ and $V_\ell$ is split, then the only one is $\PSL_2(\FF_\ell)$.    For distinct $\ell,\ell'\in D$, the non-abelian simple groups occurring in the composition series of $\Omega(V_\ell)$ and $\Omega(V_{\ell'})$  have different cardinalities (see \cite{Atlas}*{\S2.4}) and hence are not isomorphic. 
Since $H$ is a subgroup of $\prod_{\ell\in D} \Omega(V_\ell)$ such that the projection $H\to \Omega(V_\ell)$ is surjective for all $\ell\in D$, Goursat's lemma (for example, the version of Lemma~A.4 in \cite{Zywina-Maximal})  implies that that $H=\prod_{\ell\in D} \Omega(V_\ell)$.  The lemma follows since $\bbar\rho_D(\pi_1(U_{\bbar{k}})) \supseteq H$.
\end{proof}

If $R$ is a finite field, then Proposition~\ref{P:big monodromy a} follows from Lemma~\ref{L:big monodromy D} since the group $\bbar\rho_D(\pi_1(U_{\bbar{k}}))$, for a finite extension $k$ of $R$, depends only on an algebraic closure $\bbar{k}$ of $R$. \\

For the rest of the proof, we may thus assume that $R$ has characteristic $0$.   Let $R'$ be an integral domain that is an $R$-algebra.   We say that the $R'$-scheme $U_{R'}$ is \defi{nicely compactifiable} if $U_{R'}$ is open in a proper smooth $R'$-scheme $X$ and $\calD:=X-U_{R'}$ is a divisor of $X$ that has normal crossings relative to $R'$.

\begin{lemma} \label{L:resolution}
There is a finite set $S'\supseteq S$ of non-zero prime ideals of $\OO_F$ such that the $R'$-scheme $U_{R'}$ is nicely compactifiable,  where $R'=\OO_F[S'^{-1}]$.
\end{lemma}
\begin{proof}
By resolution of singularities, the variety $U_F$ over $F$ is nicely compactifiable  (note that $F$ is a field of characteristic $0$).    The lemma follows by choosing integral models for $\calD$ and $X$ and inverting enough primes.
\end{proof}

Fix a set $S'\supseteq S$ as in Lemma~\ref{L:resolution} and define $R'=\OO_F[S'^{-1}]$.   By enlarging $S'$ if necessary, we may assume that $U(\FF_\pp)$ is non-empty for all maximal ideals $\pp\notin S'$ of $\OO_F$ (since $U$ is a smooth $R$-scheme with geometric irreducible fibers of dimension at least $1$).  

\begin{lemma} \label{L:monodromy D}
Take any finite set $D\subseteq \Sigma$.    For any finite field $k$ that is an $R'$-algebra with characteristic not in $D$, the group $\bbar\rho_D(\pi_1(U_{\bbar{k}}))$ is conjugate to $G_D^g$ in $\prod_{\ell\in D} \Or(V_\ell)$.  
\end{lemma}
\begin{proof}
Take any finite set $D\subseteq \Sigma$.  Since the conclusion only involves the algebraic closure of $k$, we may assume that $k=\FF_\pp$ for a maximal ideal $\pp\notin S'$ of $\OO_F$, where $\pp$ does not divide any prime in $D$.

Let $F_\pp$ be the completion of $F$ at $\pp$ and denote by $\OO_\pp$ its valuation ring.  For an algebraic closure $\bbar{F}$ of $F$, choose an algebraic closure $\bbar{F}_\pp$ of $F_\pp$ containing $\bbar{F}$.  Since the set $U(\FF_\pp)$ is non-empty and $U$ is smooth, we have $U(\OO_\pp) \neq \emptyset$.  The $\OO_\pp$-scheme $U_{\OO_\pp}$ is nicely compactifiable since $U_{R'}$ has this property and $\OO_\pp$ is an $R'$-algebra.  So $U_{\OO_\pp}$ is open in a proper smooth $\OO_\pp$-scheme $X$ for which $\calD:=X-U_{\OO_\pp}$ is a divisor of $X$ that has normal crossings relative to $\OO_\pp$.    Let 
\[
\bbar\varrho_D \colon \pi_1(U_{\OO_\pp}) \to \prod_{\ell\in D} \Or(V_\ell)
\] 
be the representation obtained from $\bbar\rho_D$ by base extension.   By Abhyankar's Lemma \cite{MR2017446}*{XIII, 5.5}, the representation $\pi_1(U_{\FFbar_\pp})\to \prod_{\ell\in D} \Or(V_\ell)$ obtained from $\bbar{\varrho}_D$ is tamely ramified at each maximal point of the scheme $\calD_{\FFbar_\pp}$.
The Tame Specialization Theorem \cite{MR1081536}*{Theorem~8.17.14} then implies that the group $\bbar\rho_D(\pi_1(U_{\FFbar_\pp}))$ is conjugate to $\bbar\rho_D(\pi_1(U_{\bbar{F}_\pp}))$ in $\prod_{\ell\in D} \Or(V_\ell)$.  Finally, observe that $\bbar\rho_D(\pi_1(U_{\bbar{F}_\pp}))$ is conjugate to $G_D^g$ in $\prod_{\ell\in D} \Or(V_\ell)$.   
\end{proof}

\begin{lemma} \label{L:rewording monodromy}
There is a subset $\Lambda\subseteq \Sigma$ with Dirichlet density $1$ such that $G_D^g \supseteq  \prod_{\ell\in D} \Omega(V_\ell)$ for all finite subsets $D\subseteq \Lambda$.
\end{lemma}
\begin{proof}
Fix a finite field $k$ that is an $R'$-algebra.  Take $\Lambda\subseteq \Sigma$ as in Lemma~\ref{L:big monodromy D}.   For any finite $D\subseteq \Lambda$, the group $\bbar{\rho}_D(\pi_1(U_{\bbar{k}}))$ contains $\prod_{\ell\in D}\Omega(V_\ell)$ by Lemma~\ref{L:big monodromy D} and is conjugate to $G_D^g$ by Lemma~\ref{L:monodromy D}.  The lemma is now immediate.
\end{proof}

Proposition~\ref{P:big monodromy a} (in the characteristic $0$ case) and Proposition~\ref{P:big monodromy b} are now direct consequences of Lemmas~\ref{L:monodromy D} and \ref{L:rewording monodromy}.

\section{The field $K$} \label{S:the field K}

Fix notation and assumptions as in \S\ref{SS:new setup} and assume that $N$ is even.   In this section, we describe the field $K$ from \S\ref{SS:new setup}.

\begin{proposition} \label{P:K criterion}
\begin{romanenum}  
\item  \label{P:K criterion i}
There is a unique extension $K/\QQ$ with $[K:\QQ]\leq 2$ such that for all sufficiently large $\ell\in \Sigma$, $\ell$ splits in $K$ if and only if the orthogonal space $V_\ell$ is split. 
\item  \label{P:K criterion ii}
Take any $u\in U(k)$, where $k$ is a finite field that is an $R$-algebra.   If $P_u(\pm 1)\neq 0$, then
\[
K=\QQ\Big(\sqrt{(-1)^{N/2}P_u(1)P_u(-1)}\Big).
\]
\item \label{P:K criterion iii}
Take any $u\in U(k)$, where $k$ is a finite field that is an $R$-algebra. If $\varepsilon_u=1$ and $P_u(T)$ is separable, then $K=\QQ(\sqrt{\Delta_u})$ where $\Delta_u$ is the discriminant of $P_u(T)$.
\end{romanenum}
\end{proposition}
\begin{proof}

We claim that there is a point $u\in U(k)$ such that $P_u(\pm 1)\neq 0$, where $k$ is a finite field that is an $R$-algebra.   By Corollary~\ref{C:equivalent monodromy}, there is a subset $\Lambda \subseteq \Sigma$ of Dirichlet $1$ for which condition (\ref{bm-b}) of \S\ref{SSS:bm} holds.
Fix a prime $\ell \in \Lambda$ and choose an element $g\in\Omega(V_{\ell})$ such that $\det(I+ g) \neq 0$ and $\det(I- g) \neq 0$.   Since $\ell\in \Lambda$, there is a finite field $k$ that is an $R$-algebra with characteristic not equal to $\ell$ such that $\bbar\rho_{\ell}(\pi_1(U_{\bbar{k}}))\supseteq \Omega(V_{\ell})$.   By equidistribution, there is a finite extension $k'/k$ and a point $u\in U(k')$ such that $\bbar\rho_{\ell}(\Frob_{u})$ is conjugate to $g$ in $\Or(V_{\ell})$.   So $P_{u}(\pm 1) \equiv \det(I\mp g)\not\equiv 0 \pmod{\ell}$.   In particular, $P_{u}(\pm 1)\neq 0$ which proves the claim.

Now take any $u\in U(k)$ such that $P_u(\pm 1)\neq 0$, where $k$ is a finite field that is an $R$-algebra (such a point $u$ exists by the above claim).  The polynomial $P_u$ is reciprocal since $P_u(\pm 1)\neq 0$.  Define the field 
\[
K:=\QQ\Big(\sqrt{(-1)^{N/2}P_u(1)P_u(-1)}\Big).
\]
Let $\Sigma_0$ be the set of odd primes $\ell\in \Sigma$ for which $P_{u}(\pm 1)\not \equiv 0 \pmod{\ell}$.   Take any prime $\ell \in \Sigma_0$ and define $A:=\bbar{\rho}_\ell(\Frob_u) \in \Or(V_\ell)$.   We have $\det(I-TA)\equiv P_u(T) \pmod{\ell}$, so $A\in \SO(V_\ell)$ since $P_u$ is reciprocal.    By Lemma~\ref{L:disc and spin rules}(\ref{L:disc and spin rules iii}), $V_\ell$ is split if and only if $(-1)^{N/2} P_{u}(1)P_{u}(-1)$ modulo $\ell$ is a (non-zero) square in $\FF_\ell$.   So for any $\ell\in \Sigma_0$, we deduce that $\ell$ splits in $K$ if and only if $V_\ell$ is split.   In particular, for all sufficiently large $\ell\in \Sigma$, we find that $\ell$ splits in $K$ if and only if $V_\ell$ is split.   Since $\Sigma$ has density $1$, this gives a characterization of $K$ that does not depend on our choice of $u$.   Parts (\ref{P:K criterion i}) and (\ref{P:K criterion ii}) now follow since $K$ does not depend on $u$.

Finally, take any $u\in U(k)$ for which $\varepsilon_u=1$ and $P_u(T)$ is separable, where $k$ is a finite field that is an $R$-algebra. Let $\Delta_u$ be the discriminant of $P_u$; it is non-zero since $P_u$ is separable.  The polynomial $P_u$ has even degree and is reciprocal by (\ref{E:root number}) since $\varepsilon_u=1$.   By Lemma~\ref{L:basic separability}(\ref{L:basic separability i}), we have $\Delta_u  \in (-1)^{N/2} P_u(1)P_u(-1) \cdot (\QQ^\times)^2$.  Therefore, $\QQ(\sqrt{\Delta_u})=\QQ(\sqrt{(-1)^{N/2} P_u(1)P_u(-1)})$. Part (\ref{P:K criterion iii}) now follows from (\ref{P:K criterion ii}).
\end{proof}

\section{Proof of Proposition~\ref{P:criterion for sieving}}
\label{S:proof of criterion for sieving}
Fix notation and assumptions as in \S\ref{SS:new setup}.
The goal of this section is to prove the following proposition.  In this section, we prove Proposition~\ref{P:criterion for sieving} which  will be used to apply the sieve theory in the proofs of Theorems~\ref{T:main 1} and \ref{T:main 2}.   

\subsection{Big subgroups of $W_{2n}$}
We first give a criterion to prove that a subgroup of $W_{2n}$ contains $W_{2n}^+$.  We shall assume that $n\geq 2$; the case $n=1$ is not interesting since $W_{2}^+=1$.

We may view $W_{2n}$ as a subgroup of the group of permutations $\mathfrak{S}_X$ of the set $X=\{\pm e_1,\ldots, \pm e_n\}$.  Let 
 $\varepsilon_1\colon W_{2n}\to \{\pm 1\}$ be the homomorphism obtained by composing the inclusion $W_{2n}\hookrightarrow \mathfrak{S}_X$ with the signature map.  The kernel of $\varepsilon_1$ is the subgroup $W_{2n}^+$.   By considering the action of $W_{2n}$ on the $n$ pairs $p_i:=\{e_i,-e_i\}$ with $1\leq i \leq n$, we obtain a homomorphism $\varphi\colon W_{2n}\to \mathfrak{S}_n$.  Let $\varepsilon_2\colon W_{2n}\to \{\pm 1\}$ be the homomorphism obtained by composing $\varphi$ with the signature map.   

\begin{lemma} \label{L:group theory for W2n}  
Let $G$ be a subgroup of $W_{2n}$.  Suppose that there exist $g_1$, $g_2$, $g_3$, $g_4$ and $g_5$ in $G$ such that the following hold:
\begin{itemize}
\item $\varphi(g_1)\in \mathfrak{S}_n$ is an $n$-cycle,
\item $\varphi(g_2)\in \mathfrak{S}_n$ is a $p$-cycle for some prime $p>n/2$,
\item $\varphi(g_3)\in \mathfrak{S}_n$ is a transposition,
\item $g_4\in \mathfrak{S}_X$ satisfies $\varphi(g_4)=1$ and is a product of one or two disjoint transpositions,
\item $\varepsilon_1(g_5)\varepsilon_2(g_5)=-1$.
\end{itemize}
Then $G$ equals $W_{2n}^+$ or $W_{2n}$.
\end{lemma}
\begin{proof}
A lemma of Brauer, see \cite{MR0332694}*{p.98}, says that $\mathfrak{S}_n$ has no proper transitive subgroups that contain a transposition and a cycle of prime order greater than $n/2$.  The properties of $g_1$, $g_2$ and $g_3$ thus ensure that $\varphi(G)=\mathfrak{S}_n$.   

Let $H$ be the kernel of $\varphi\colon W_{2n}\to \mathfrak{S}_n$; these are the permutations of $X$ that fix all the pairs $p_i=\{e_i,-e_i\}$.    Let $H^+$ be the kernel of $\varepsilon_1|_H\colon H \to \{\pm 1\}$. \\

First suppose that $g_4$ is the product of two disjoint transpositions.  We have $g_4 \in H$ since $\varphi(g_4)=1$.   Without loss of generality, we may assume that $g_4$ interchanges $e_1$ and $-e_1$,  interchanges $e_2$ and $-e_2$, and fixes all the other $\pm e_j$.   

Take any $1\leq i \leq n$.   Since $\varphi(G)=\mathfrak{S}_n$, there is an element $\sigma \in G$ satisfying $\varphi(\sigma)=(1i)$.    Then $\sigma g_4 \sigma^{-1} \in G$ interchanges $e_2$ and $-e_2$, $e_i$ and $-e_i$, and fixes all the other $\pm e_j$.   Therefore, $h_i:=\sigma g_4 \sigma^{-1}g_4^{-1} \in G$ interchanges $e_1$ and $-e_1$, $e_i$ and $-e_i$, and fixes all the other $\pm e_j$.   Observe that $h_i$ is in the commutator subgroup $[G,G]$ of $G$.

The elements of $H$ that are the product of two disjoint transpositions are precisely the elements $h_i$ with $1<i\leq n$ or $h_i h_j$ with $1< i<j\leq n$.   We thus have $H^+\subseteq G$ since the group $H^+$ is generated by the elements in $H$ that are the product of two disjoint transpositions.   Moreover,  $H^+\subseteq [G,G]$.

Since $\varphi(G)=\mathfrak{S}_n$,  the group $\varphi([G,G])$ equals the commutator subgroup of $\mathfrak{S}_n$; this is the alternating group $\mathfrak{A}_n$ since $n\geq 2$.   Therefore, the cardinality of the group $[G,G]$ is divisible by $|H^+|\cdot |\mathfrak{A}_n| = 2^{n-1}\cdot n!/2 = 2^{n-2} n!$.  We have $|W_{2n}|=2^n n!$, so $[W_{2n}:[G,G]]\leq 4$.  Since $[G,G]$ is contained in the commutator subgroup $[W_{2n},W_{2n}]$, we have 
\[
m:=[W_{2n}: [W_{2n},W_{2n}]]\leq [W_{2n}:[G,G]] \leq 4.
\]    
However, we have $m\geq 4$ since the quotient of $W_{2n}$ by $\ker(\varepsilon_1) \cap \ker(\varepsilon_2)$ is isomorphic to $\{\pm 1\}\times \{\pm 1\}$.  Therefore, $m=4$ and hence
\[
[G,G] = [W_{2n},W_{2n}] = \ker(\varepsilon_1) \cap \ker(\varepsilon_2).
\]
We have $\varepsilon_2(G)=\{\pm 1\}$ since $\varphi(G)=\mathfrak{S}_n$.   Therefore, $G$ must be one of the groups $\ker(\varepsilon_1\varepsilon_2)$, $\ker(\varepsilon_1)=W_{2n}^+$ or $W_{2n}$.  The existence of $g_5$ rules out the case $G=\ker(\varepsilon_1\varepsilon_2)$.  Therefore, $G$ equals $W_{2n}^+$ or $W_{2n}$.\\

Now suppose that $g_4\in G$ is a transposition.   We have $g_4\in H$ since $\varphi(g_4)=1$.   Using that $\varphi(G)=\mathfrak{S}_n$, an argument similar to the one above shows that $G$ contains every transposition in $H$.   We have $H\subseteq G$ since $H$ is generated by transpositions.   Therefore, $G$ is a group of order $|H|\cdot |\mathfrak{S}_n| = 2^n n!$.   Since $G$ has the same cardinality as $W_{2n}$, we conclude that $G=W_{2n}$.
\end{proof}

\begin{remark} 
Note that in part (ii) of Lemma~4.4 in \cite{MR2539184}, which is an analogue of our Lemma~\ref{L:group theory for W2n}, one needs to add another condition to rule out the case where the subgroup of $W_{2n}$ is $\ker(\epsilon_1 \epsilon_2)$.
\end{remark}

\subsection{A criterion for a maximal Galois group} \label{SS:criterion for maximal Galois}

Fix an integer $n\geq 1$ and let $\FF$ be a finite field of odd characteristic.  Let $\calP_n(\FF)$ be the set of monic polynomials $h\in \FF[T]$ of degree $n$ which are separable and satisfy $h(\pm2)\neq 0$.  

\noindent If $n\geq 2$, define the following sets:
\begin{itemize}
\item  
Let $H_{n,1}(\FF)$ be the set of irreducible $h\in \calP_n(\FF)$.
\item  
Let $H_{n,2}(\FF)$ be the set of $h\in \calP_n(\FF)$ that have an irreducible factor whose degree is a prime greater than $n/2$.
\item  
Let $H_{n,3}(\FF)$ be the set of $h\in \calP_n(\FF)$ that factor as a product of an irreducible polynomial of degree 2 and irreducible polynomials of odd degree.
\item  
Let $H_{n,4}(\FF)$ be the set of $h\in \calP_n(\FF)$ that have no irreducible factors of even degree, and for which the polynomial $T^n h(T+1/T)$ is the product of one or two quadratic irreducible polynomials and irreducible polynomials of odd degree.
\item  
Let $H_{n,5}(\FF)$ be the set of $h\in \calP_n(\FF)$ such that the polynomial $h(T)\cdot T^n h(T+1/T)$ has an odd number of irreducible factors of even degree (counted with multiplicity).
\item  
Let $H_{n,6}(\FF)$ be the set of $h\in \calP_n(\FF)$ such that the polynomial $T^n h(T+1/T)$ is the product of a quadratic irreducible polynomial and irreducible polynomials of odd degree.
\end{itemize}
If $n=1$, define $H_{n,i}(\FF)=\calP_1(\FF)$ for all $1\leq i \leq 5$ and let $H_{n,6}(\FF)$ be the set of $h\in \calP_1(\FF)$ such that the quadratic polynomial $T h(T+1/T)$ is irreducible.  

For $1\leq i\leq 6$, we define $F_{2n,i}(\FF)$ to be the set of polynomials $T^n h(T+1/T)$ with $h\in H_{n,i}(\FF)$; they are monic, reciprocal and have degree $2n$.   By Lemma~\ref{L:basic separability}(\ref{L:basic separability i}), the condition that $h$ is separable and $h(\pm 2)\neq 0$ ensures that each polynomials $f\in F_{2n,i}(\FF)$ is separable and $f(\pm 1)\neq 0$.  \\

The above definitions are justified by the following criterion. 

\begin{proposition} \label{P:maximality criterion}
Fix a monic, reciprocal and separable polynomial $f\in \QQ[T]$ of even degree $2n\geq 2$.  Let $\Delta$ be the discriminant of $f$.    Denote by $\Gal(f)$ the Galois group of a splitting field of $f$ over $\QQ$.    
Assume that for each $1\leq i\leq 5$, there is an odd prime $\ell$ such that the coefficients of $f$ are integral at $\ell$ and $f \bmod{\ell} \in \FF_{\ell}[T]$ lies in $F_{2n,i}(\FF_\ell)$.    
\begin{romanenum}
\item  \label{P:maximality criterion i}
If $\Delta$ is a square in $\QQ$, then $\Gal(f)\cong W_{2n}^+$.
\item  \label{P:maximality criterion ii}
If $\Delta$ is a non-square in $\QQ$, then $\Gal(f) \cong W_{2n}$.
\item  \label{P:maximality criterion iii}
If there is an odd prime $\ell$ such that the coefficients of $f$ are integral at $\ell$ and $f \bmod{\ell} \in \FF_{\ell}[T]$ lies in $F_{2n,6}(\FF_\ell)$, then $\Gal(f) \cong W_{2n}$.
\end{romanenum}
\end{proposition}
\begin{proof}
If $n=1$, then (\ref{P:maximality criterion i}) and (\ref{P:maximality criterion ii}) are immediate since $f$ is a separable quadratic polynomial and the groups $W_{2}^+$ and $W_2$ have cardinality $1$ and $2$, respectively.  So assume that $n\geq 2$.  As in \S\ref{SS:W def}, we have an injective homomorphism 
\[
\psi\colon \Gal(f) \hookrightarrow W_{2n}.
\]
Take any prime $\ell$ for which the coefficients of $f$ are integral at $\ell$ and $f$ modulo $\ell$ is separable with the same degree as $f$.   Then $\psi$ is unramified at $\ell$ and the cycle type of $\psi(\Frob_\ell)$ in $\mathfrak{S}_X$ is given by the degrees of the irreducible factors of $f$ modulo $\ell$.    The cycle type of $\varphi(\psi(\Frob_\ell))$ in $\mathfrak{S}_n$ is given by the degrees of the irreducible factors of $h$ modulo $\ell$.

\begin{itemize}
\item 
Since $h \bmod{\ell_1}$ is irreducible in $\FF_{\ell_1}[T]$, we find that $\varphi(\psi(\Frob_{\ell_1}))$ is a $n$-cycle in $\mathfrak{S}_n$.  
\item
Since $h \bmod{\ell_2} \in \FF_{\ell_2}[T]$ has an irreducible factor of prime degree $p>n/2$, we find that some power of $\varphi(\psi(\Frob_{\ell_2}))$ is a $p$-cycle in $\mathfrak{S}_n$.  
\item
Since $h \bmod{\ell_3} \in \FF_{\ell_3}[T]$ is the product of an irreducible quadratic polynomial and irreducibles of odd degree, we find that some power of $\varphi(\psi(\Frob_{\ell_3}))$ is a transposition in $\mathfrak{S}_n$.  
\item
Since $h \bmod{\ell_4}$ has no irreducible factors of even degree and $f\bmod{\ell_4}$ is the product of one or two quadratic irreducible polynomials and irreducible polynomials of odd degree, we find that there is a power $g$ of $\psi(\Frob_{\ell_4})$ such that $\varphi(g)=1$ and $g$ is a product of one or two disjoint transpositions in $\mathfrak{S}_X$.  \item
Since $hf \bmod{\ell_5}$ has an odd number of irreducible factors of even degree, we find that \[
\epsilon_1(\psi(\Frob_{\ell_5}))\epsilon_2(\psi(\Frob_{\ell_5}))= -1.
\]  
\end{itemize}
By Lemma~\ref{L:group theory for W2n}, the group $\psi(\Gal(f))$ is either $W_{2n}^+$ or $W_{2n}$.  The image of $\psi$ is a subgroup of $W_{2n}^+$ if and only if the discriminant $\Delta$ of $f$ is a square in $F$.  So $\psi(\Gal(f))=W_{2n}^+$ if $\Delta$ is a square in $F$ and $\psi(\Gal(f))=W_{2n}$ if $\Delta$ is not a square in $F$.  This proves parts (\ref{P:maximality criterion i}) and (\ref{P:maximality criterion ii}).

Finally, suppose there is a prime $\ell$ as in the statement of part (\ref{P:maximality criterion iii}).  Then $\psi$ is unramified at $\lambda$ and the permutation $\psi(\Frob_\lambda)$ in $\mathfrak{S}_X$ is the product of disjoint cycles where one is a transposition and the rest have odd length.   Therefore, $\varepsilon_1(\psi(\Frob_\lambda))=-1$ and hence $\psi(\Gal(f))\neq W_{2n}^+$.  So, $\psi(\Gal(f))=W_{2n}$.
\end{proof}

For cosets $\alpha,\beta\in \FF^\times/(\FF^\times)^2$, we define $F_{2n,i}^{\alpha,\beta}(\FF)$ to be the set of $f\in F_{2n,i}(\FF)$ such that $f(\pm 1)\neq 0$, $f(1) \in \alpha$, $f(-1)\in \beta$, and $f$ has at most eight irreducible factors.   
 The follows lower bounds for the cardinality of $F_{2n,i}^{\alpha,\beta}(\FF)$ will be important later for counting certain subsets of orthogonal groups.
 
\begin{proposition} \label{P:H bounds}
Fix $\alpha,\beta\in \FF^\times/(\FF^\times)^2$ and an integer $1\leq i \leq 6$.  Assume that $\alpha\beta\neq (-1)^n(\FF^\times)^2$ if $i=6$.  Then
\begin{equation} \label{E:Hni inequality}
|F_{2n,i}^{\alpha,\beta}(\FF)|\geq \frac{c}{n^2} q^n\big(1 + O(1/q)\big),
\end{equation}
where $q$ is the cardinality of $\FF$, and the constant $c>0$ and the implicit constant are absolute. 
\end{proposition}

Before proving the proposition, we need a lemma.  For $m\geq 1$ and cosets $\alpha,\beta\in \FF^\times/(\FF^\times)^2$, let $\calI^{\alpha,\beta}_m$ be the set of irreducible $h\in \calP_m(\FF)$ such that $h(2)\in \alpha$ and $h(-2)\in \beta$.  Set $\calI^{\alpha,\beta}_0=\{1\}$.   

\begin{lemma}  \label{L:irred approx}
For $m\geq 1$, we have $|\calI_m^{\alpha,\beta}|=\frac{1}{4m}\big(q^m + O(q^{m/2}) \big)$,  where the implicit constant is absolute. 
\end{lemma}
\begin{proof}
Set $\FF_q=\FF$. Choose elements $a\in \alpha$ and $b\in \beta$.  The map
\begin{align*}
\{ \zeta \in \FF_{q^m}: \FF_q(\zeta)=\FF_{q^m}  \} &\to \{h \in \FF_q[T] : h \text{ monic and irreducible of degree } m\}
\end{align*}
defined by $ \zeta \mapsto N_{\FF_{q^m}/\FF_q}(T-\zeta)$ is surjective and $m$-to-$1$.   

Fix $\zeta\in \FF_{q^m}$ such that $\FF_q(\zeta)=\FF_{q^m}$, and set $h(T)=N_{\FF_{q^m}/\FF_q}(T-\zeta)$. We have $h(\pm 2)\neq 0$ if and only if $\zeta\neq \pm 2$.  Since $N_{\FF_{q^m}/\FF}$ induces an isomorphism $\FF_{q^m}^\times/(\FF_{q^m}^\times)^2 \to \FF_{q}^\times/(\FF_{q}^\times)^2$, wcce have $h(2) \in \alpha$ and $h(-2)\in \beta$ if and only if $2-\zeta \in a (\FF_{q^m}^\times)^2$ and $-2-\zeta \in b (\FF_{q^m}^\times)^2$.  Therefore,
\begin{align*}
m |\calI_m^{\alpha,\beta}(\FF)|&= |\{ \zeta \in \FF_{q^m}-\{\pm 2\}:  \FF_{q}(\zeta)=\FF_{q^m}, \, 2-\zeta = a x^2 \text{ and } -2-\zeta=by^2 \text{ for some } x,y\in \FF_{q^m} \}|\\
&= \tfrac{1}{4}  |\{ (x,y) \in \FF_{q^m}^2:  ax^2 -by^2= 4 \}| + O(|\{\zeta\in \FF_{q^m}: \FF_q(\zeta)\neq \FF_{q^m}\}| + 1).
\end{align*}
The projective closure of the plane curve $ax^2 -by^2= 4$ is smooth of genus $0$.  So $m |\calI_m^{\alpha,\beta}(\FF)|$ equals $q^{m}/4 + O(|\{\zeta\in \FF_{q^m}: \FF_q(\zeta)\neq \FF_{q^m}\}| + 1)$.  Finally, note that
\[
|\{\zeta\in \FF_{q^m}: \FF_q(\zeta)\neq \FF_{q^m}\}|\leq \sum_{d|m, d<m} |\FF_{q^d}|\leq \sum_{d\leq m/2} q^d =(q^{\lfloor m/2 \rfloor+1}-1)/(q-1) = O( q^{m/2}).  \qedhere
\]
\end{proof}

\begin{proof}[Proof of Proposition~\ref{P:H bounds}]  
For cosets $\alpha,\beta\in \FF^\times/(\FF^\times)^2$, we define $H_{n,i}^{\alpha,\beta}(\FF)$ to be the set of $h\in H_{n,i}(\FF)$ such that $h(\pm 2)\neq 0$, $h(2)\in \alpha$, $h(-2)\in \beta$, and $h$ has at most four irreducible factors.   By Lemma~\ref{L:basic separability}(\ref{L:basic separability ii}), the polynomial $T^{n} h(T+1/T) \in \FF[T]$ has at most eight irreducible factors for all $h\in H_{n,i}(\FF)$.  We thus have an injective map
\[
H_{n,i}^{\alpha,\beta}(\FF) \hookrightarrow  F_{2n,i}^{\alpha,(-1)^n \beta}(\FF),\quad h \mapsto T^{n} h(T+1/T).
\]
It thus suffices to show that 
\[
|H_{n,i}^{\alpha,\beta}(\FF)|\geq \frac{c}{n^2} \cdot (q^n + O(q^{n-1})),
\] 
where $c>0$ and the implicit constant are absolute, and $\alpha\beta \neq (\FF^\times)^2$ if $i=6$.
\\  
  
 The following inclusions involving $H_{n,i}^{\alpha,\beta}(\FF)$ make use of Lemma~\ref{L:basic separability}(\ref{L:basic separability ii}) when $i \in\{4,5,6\}$.    Let $\gamma$ be the non-identity coset of $\FF^\times/(\FF^\times)^2$.      When $n=1$, we have  $H_{n,i}^{\alpha,\beta}(\FF)=\calI_n^{\alpha,\beta}$ for $1\leq i \leq 5$.  We also have $H_{1,6}^{\alpha,\beta}(\FF) = \calI_1^{\alpha,\beta}$ when $\alpha\beta  =\gamma$.  
 
 \noindent
 Now suppose that $n\geq 2$.
\begin{itemize}
\item 
 We have $\calI_n^{\alpha,\beta}=H_{n,1}^{\alpha,\beta}(\FF)$.   
\item 
 By Bertrand's postulate, there exists a prime $n/2<p\leq n$ and hence  
\[
\{ h_1 h_2 : (h_1,h_2) \in \calI_p^{\alpha,\beta} \times \calI_{n-p}^{1,1} \text{ and } h_1\neq h_2 \} \subseteq H_{n,2}^{\alpha,\beta}(\FF).
\]
\item
If $n$ is odd, then  $\{h_1h_2: (h_1,h_2) \in \calI_{2}^{\alpha,\beta} \times \calI_{n-2}^{1,1}\}  \subseteq H_{n,3}^{\alpha,\beta}(\FF)$.

\noindent 
If $n\geq 4$ is even, then  $\{ h_1 h_2 h_3: (h_1,h_2,h_3) \in \calI_{2}^{\alpha,\beta}\times \calI_{1}^{1,1} \times \calI_{n-3}^{1,1} \text{ and }  h_2\neq h_3\} \subseteq H_{n,3}^{\alpha,\beta}(\FF)$.

\noindent 
If $n=2$, then  $\calI_{2}^{\alpha,\beta} \subseteq H_{n,3}^{\alpha,\beta}(\FF)$.
\item
If $\alpha\beta=1$ and $n$ is odd, then 
\[
\{h_1 h_2 h_3 : (h_1,h_2,h_3)\in \calI_1^{\alpha,\beta\gamma}\times \calI_1^{1,\gamma} \times \calI_{n-2}^{1,1} \text{ and } h_1,h_2,h_3 \text{ distinct} \} \subseteq H_{n,4}^{\alpha,\beta}(\FF)
\]
and $\{h_1 h_2: (h_1,h_2) \in \calI_2^{\alpha,\beta}\times \calI_{n-2}^{1,1} \} \subseteq H_{n,5}^{\alpha,\beta}(\FF)$.
\item
If $\alpha\beta=1$ and $n=2$, then 
\[
\{h_1 h_2: (h_1,h_2)\in \calI_1^{\alpha,\beta\gamma} \times \calI_1^{1,\gamma} \text{ and } h_1\neq h_2 \} \subseteq H_{n,4}^{\alpha,\beta}(\FF)
\]
and $\calI_2^{\alpha,\beta}\subseteq H_{n,5}^{\alpha,\beta}(\FF)$.  
\item
If $\alpha\beta=1$ and $n\geq 4$ is even, then 
\[
\{h_1 h_2 h_3 h_4 : (h_1,h_2,h_3,h_4)\in \calI_1^{\alpha,\beta\gamma} \times \calI_1^{1,\gamma}\times \calI_1^{1,1}\times \calI_{n-3}^{1,1} :  h_1,h_2,h_3,h_4 \text{ distinct}\} \subseteq H_{n,4}^{\alpha,\beta}(\FF)
\]
and $ \{h_1 h_2 h_3: (h_1,h_2,h_3)\in \calI_2^{\alpha,\beta}\times \calI_1^{1,1}\times \calI_{n-3}^{1,1} : h_2\neq h_3 \} \subseteq H_{n,5}^{\alpha,\beta}(\FF)$.
\item
If $\alpha\beta=\gamma$ and $n$ is odd, then 
\[
\{h_1 h_2 h_3 : (h_1,h_2,h_3)\in \calI_1^{\alpha,\beta}\times \calI_1^{1,1} \times \calI_{n-2}^{1,1},\,\, h_1,h_2,h_3 \text{ distinct} \}
\]
is a subset of $H_{n,4}^{\alpha,\beta}(\FF)$ and $H_{n,5}^{\alpha,\beta}(\FF)$.  
\item 
If $\alpha\beta=\gamma$ and $n$ is even, then 
\[
\{h_1h_2 : (h_1,h_2)\in \calI_1^{\alpha,\beta}\times \calI_{n-1}^{1,1} \text{ and } h_1\neq h_2  \}
\]
is a subset of $H_{n,4}^{\alpha,\beta}(\FF)$ and $H_{n,5}^{\alpha,\beta}(\FF)$.  

\item
If $\alpha\beta=\gamma$ and $n$ is odd, then  
\[
\{h_1h_2h_3 : (h_1,h_2,h_3)\in \calI_1^{\alpha,\beta}\times \calI_1^{1,1}\times \calI_{n-2}^{1,1} \text{ and } h_2\neq h_3  \} \subseteq H_{n,6}^{\alpha,\beta}(\FF).
\]  
\noindent 
If $\alpha\beta=\gamma$ and $n$ is even, then  $\{h_1h_2 : (h_1,h_2)\in \calI_1^{\alpha,\beta}\times \calI_{n-1}^{1,1} \} \subseteq H_{n,6}^{\alpha,\beta}(\FF)$.  
\end{itemize}

The proposition follows immediately from the above inclusions and Lemma~\ref{L:irred approx}.
\end{proof}

\subsection{Proof of Proposition~\ref{P:criterion for sieving}}

First fix a prime $\ell \in \Sigma$.  Let $\kappa$ be any coset of $\Omega(V_\ell)$ in $\Or(V_\ell)$.  There are unique $\varepsilon \in \{\pm 1\}$ and $\delta \in \FF_\ell^\times/(\FF_\ell^\times)^2$ such that $\det(\kappa)=\{\varepsilon\}$ and $\spin(\kappa)=\{\delta\}$.     

Take any $1\leq i \leq 6$.  We now define a subset $C_i(\kappa) \subseteq \kappa$ that is stable under conjugacy by $\Or(V_\ell)$ (the sets of polynomials $F_{2n,i}(\FF_\ell)$ are those from \S\ref{SS:criterion for maximal Galois}):
\begin{itemize}
\item 
If $N$ is odd, let $C_i(\kappa)$ be the set of $A\in \kappa$ such that $\det(I-AT)/(1-\varepsilon T)$ lies in $F_{N-1,i}(\FF_\ell)$.
\item 
If $N$ is even and $\varepsilon=-1$, let $C_i(\kappa)$ be the set of $A\in \kappa$ such that $\det(I-AT)/(1-T^2)$ lies in $F_{N-2,i}(\FF_\ell)$.
\item 
If $N$ is even, $\varepsilon=1$ and $i\neq 6$, let $C_i(\kappa)$ be the set of $A\in \kappa$ such that $\det(I-AT)$ lies in $F_{N,i}(\FF_\ell)$. 
\item
If $N$ is even, $\varepsilon=1$ and $i= 6$, define $C_i(\kappa)=\kappa$. 
\end{itemize}

\begin{lemma} \label{L:main lower}
There is a positive constant $c$  such that 
\[
\frac{|C_i(\kappa)|}{|\Omega(V_\ell)|} \geq \frac{c}{N^2} \cdot (1+O(1/\ell))
\] 
holds for all $1\leq i \leq 6$, where $c$ and the implicit constant are absolute.
\end{lemma}
\begin{proof}

\noindent $\bullet$ \emph{Suppose that $N$ is odd.}\\
\noindent  Fix any $\alpha,\beta\in \FF_\ell^\times/(\FF_\ell^\times)^2$ satisfying $\alpha\beta\neq (-1)^{(N-1)/2}(\FF_\ell^\times)^2$ such that $\delta=\beta$  if $\varepsilon=1$ and $\delta=\alpha \disc(V_\ell)$ if $\varepsilon=-1$.     Take any $f\in F_{N-1,i}^{\alpha,\beta}(\FF_\ell)$.  Proposition~\ref{P:conjugacy for N odd} implies that 
\[
C_f:=\{A\in \Or(V_\ell): \det(I-AT)=(1-\varepsilon)f(T)\}
\] 
is a conjugacy class of $\Or(V_\ell)$ and 
\[
|C_f|/|\Omega(V_\ell)| \geq 2 \ell^{-(N-1)/2}(1+O(1/\ell))
\] 
with an absolute implicit constant.  Note that for the constant to be absolute, we have used that $f$ has at most eight irreducible factors.   

By Proposition~\ref{P:conjugacy for N odd} and our choice of $\alpha$ and $\beta$,  we have $\det(C_f)=\{\varepsilon\}$ and $\spin(C_f)=\{\delta\}$, and thus $C_f\subseteq \kappa$.  Therefore, 
\[
{|C_i(\kappa)|}/{|\Omega(V_\ell)|} \geq |F_{N-1,i}^{\alpha,\beta}(\FF_\ell)|\cdot 2 \ell^{-(N-1)/2}(1+O(1/\ell)) \gg {1}/{N^2}\cdot(1+O(1/\ell))
\]
with absolute constants, where the last inequality uses Proposition~\ref{P:H bounds}.

\noindent $\bullet$  \emph{Suppose that $N$ is even and $\varepsilon=-1$.}\\
\noindent 
Take any $\alpha,\beta\in \FF_\ell^\times/(\FF_\ell^\times)^2$ and any $f\in F_{N-2,i}^{\alpha,\beta}(\FF_\ell)$.  Proposition~\ref{P:conjugacy for N even 2} implies that 
\[
C_f:=\{A\in \Or(V_\ell): \det(I-AT)=(1-T^2)f(T) \text{ and } \spin(A)=\delta\}
\] 
is a conjugacy class of $\Or(V_\ell)$ and $|C_f|/|\Omega(V_\ell)| \geq \ell^{-(N-2)/2}(1+O(1/\ell))$ with an absolute constant.    We have $C_f \subseteq \kappa$, so 
\[
{|C_i(\kappa)|}/{|\Omega(V_\ell)|} \geq |F_{N-2,i}^{\alpha,\beta}(\FF_\ell)|\cdot  \ell^{-(N-2)/2}(1+O(1/\ell)) \gg {1}/{N^2}\cdot (1+O(1/\ell))
\]
with absolute constants, where the last inequality uses Proposition~\ref{P:H bounds}.

\noindent $\bullet$  \emph{Suppose that $N$ is even and $\varepsilon=1$.}\\
\noindent 
If $i=6$, we have $C_i(\kappa)=\kappa$ and the lemma is easy.   

Now suppose that $1\leq i\leq 5$. Take $\alpha,\beta\in \FF_\ell^\times/(\FF_\ell^\times)^2$ such that $\delta=\beta$ and $\disc(V_\ell)=\alpha\beta$.   Take any $f\in F_{N,i}^{\alpha,\beta}(\FF_\ell)$.  Proposition~\ref{P:conjugacy for N even}(\ref{P:conjugacy for N even ii}) implies that 
\[
C_f:=\{A\in \Or(V): \det(I-AT)=f(T)\}
\] 
is a conjugacy class of $\Or(V_\ell)$  and $|C_f|/|\Omega(V_\ell)| \geq 4\ell^{-N/2}(1+O(1/\ell))$ with an absolute constant.   By Proposition~\ref{P:conjugacy for N even}(\ref{P:conjugacy for N even ii}), we have $\det(C_f)=\{1\}$ and $\spin(C_f)=\{f(-1)(\FF_\ell^\times)^2\}=\{\beta\}$.   Therefore, 
\[
{|C_i(\kappa)|}/{|\Omega(V_\ell)|} \geq |F_{N,i}^{\alpha,\beta}(\FF_\ell)|\cdot 4\ell^{-(N-2)/2}(1+O(1/\ell)) \gg {1}/{N^2}\cdot (1+O(1/\ell))
\]
with absolute constants, where the last inequality uses Proposition~\ref{P:H bounds}  (recall that $i\neq 6$).
\end{proof}

For any integer $1\leq i\leq 6$, define 
\[
C_i(V_\ell):=\bigcup_{\kappa} C_i(\kappa),
\] 
where the union is over the four cosets $\kappa$ of $\Omega(V_\ell)$ in $\Or(V_\ell)$.  The set $C_i(V_\ell)$ is stable under conjugation by $\Or(V_\ell)$.  By Lemma~\ref{L:main lower},  we have $|C_i(V_\ell)\cap \kappa|/|\kappa| = |C_i(\kappa)|/|\Omega(V_\ell)| \gg {1}/{N^2}\cdot (1+O(1/\ell))$ for each coset $\kappa$.   There are thus positive absolute constants $c_1$ and $c_2$ such that if $\ell \geq c_1$, then $|C_i(V_\ell)\cap \kappa|/|\kappa| \geq c_2/N^2$ for all cosets $\kappa$ of $\Omega(V_\ell)$ in $\Or(V_\ell)$.  \\

We have now constructed sets $\{C_i(V_\ell)\}_{\ell\in \Sigma}$ for all $1\leq i \leq 6$.  It thus remains to verify that (\ref{P:criterion for sieving c}) holds with these sets.    Take any $u\in U(k)$, where $k$ is a finite field that is an $R$-algebra.  Suppose that for each $1\leq i \leq 6$, there is a prime $\ell_i \in \Sigma$ for which $\bbar\rho_{\ell_i} (\Frob_u) \in C_i(V_{\ell_i})$.   

Let $f_u$ be the polynomial obtained from $P_u$ by the formula (\ref{E:initial f def}).   Set $n=\deg(f_u)/2$.   If $N$ is odd, we have $N=2n+1$.   If $N$ is even, then $N$ is $2n$ or $2n+2$ when $\varepsilon_u$ is $1$ or $-1$, respectively.  For each $1\leq i \leq 6$, with $i\neq 6$ when $N$ is even and $\varepsilon_u=1$, the inclusion $\bbar\rho_{\ell_i} (\Frob_u) \subseteq C_i(V_{\ell_i})$ implies that $f_u$ modulo $\ell_i$ lies in $F_{2n,i}(\FF_\ell)$.   If $N$ is odd or $\varepsilon_u=-1$, Proposition~\ref{P:maximality criterion}(\ref{P:maximality criterion iii}) implies that the Galois group of $f_u$, and hence also $P_u$, is isomorphic to $W_{2n}$.

Finally, suppose that $N$ is even and $\varepsilon_u=1$.   The polynomial $P_u=f_u$ is separable since its reduction modulo $\ell_1$ is separable.  By Proposition~\ref{P:K criterion}(\ref{P:K criterion iii}), the discriminant of $P_u$ is a square in $\QQ$ if and only if $K=\QQ$.  From Proposition~\ref{P:maximality criterion}(\ref{P:maximality criterion i}) and (\ref{P:maximality criterion ii}), we deduce that the Galois group of $P_u$ is isomorphic to $W_{N}^+$ if $K=\QQ$ and $W_N$ if $K\neq \QQ$.  

\section{Proof of Theorem~\ref{T:main 1}} \label{S:proof of main 1}
Fix notations and assumptions as in \S\ref{SS:new setup}.   

Suppose that $R$ has characteristic $0$ and hence $R=\ZZ[S^{-1}]$ for a finite set $S$ of non-zero prime ideals of $\OO_F$.    Take  $S'\supseteq S$ and $\Lambda\subseteq \Sigma$ as in Proposition~\ref{P:big monodromy b}.    For the finite number of $\pp\in S'-S$, we can base extend everything to $\FF_\pp$ and the assumptions of \S\ref{SS:new setup} still hold with the base ring $\FF_\pp$.   So Theorem~\ref{T:main 1} in the finite field case, would imply that $\delta(k)\to 1$ as we vary over all finite extensions $k$ of $\FF_\pp$ for some $\pp\in S'-S$.  So assuming Theorem~\ref{T:main 1} in the finite field case, we can reduce to the case where we base extend everything to $R[S'^{-1}]=\ZZ[S'^{-1}]$.   So without loss of generality, we may assume that Proposition~\ref{P:big monodromy b} holds with $S'= S$ and $\Lambda\subseteq \Sigma$ a set of Dirichlet density $1$.  By replacing $\Lambda$ by an appropriate subset with Dirichlet density $1$, we may further assume that it satisfies Proposition~\ref{P:big monodromy a}.  

If $R$ is a finite field, we take $\Lambda$ as in Proposition~\ref{P:big monodromy a}.  

 Let $c_1 \geq 5$ and $c_2$ be positive absolute constants as in Proposition~\ref{P:criterion for sieving}(\ref{P:criterion for sieving b}).  By replacing $c_2$ with a smaller value, we may assume that $0<c_2/N^2<1$.   By removing a finite number of primes from $\Lambda$, we may also assume that each prime $\ell\in \Lambda$ is greater than $c_1$.\\

Take any $\varepsilon>0$.   We will prove that 
\begin{align} \label{E:delta goal}
1-\delta(k)<\varepsilon+O(|k|^{-1/2})
\end{align}
holds for all finite fields $k$ that are $R$-algebras, where the implicit constant does not depend on $k$.   This will imply that $0\leq \limsup_{k,\,\#k\to \infty} (1-\delta(k)) \leq \varepsilon$ 
where $k$ varies over finite fields that are $R$-algebras with increasing cardinality.  Since $\varepsilon>0$ was arbitrary, we will then have $\lim_{k,\,\#k\to \infty} \delta(k)=1$ which will complete the proof of Theorem~\ref{T:main 1}.

Since $0<c_2/N^2<1$, we can choose an integer $M\geq 1$ satisfying $(1-c_2/N^2)^{M}<\varepsilon/6$.   Since $\Lambda$ is infinite, we can choose a finite set $D\subseteq \Lambda$ of cardinality $M$.   It suffices to prove that (\ref{E:delta goal}) holds when the characteristic of $k$ does not lie in $D$ (we can then repeat the proof with a different set $D\subseteq\Lambda$ of cardinality $M$ that is disjoint from the original one).\\

Take any finite field $k$ that is an $R$-algebra and whose characteristic does not lie in $D$.   If $U(k)$ is empty, then $|k|$ is bounded and hence (\ref{E:delta goal}) holds for an appropriate implicit constant.  We may thus assume that $U(k)$ is non-empty

For each integer $1\leq i \leq 6$, define the set
\[
\calS_i = \{ u \in U(k) : \bbar\rho_{\ell}(\Frob_u) \not\subseteq C_i(V_\ell) \text{ for all }\ell\in D \},
\]
where the sets $C_i(V_\ell)$ are from Proposition~\ref{P:criterion for sieving}.   Proposition~\ref{P:criterion for sieving}(\ref{P:criterion for sieving c}) implies that
\[
\{u \in U(k): P_u(T) \text{ does not satisfy (\ref{E:Galois specific})}\} \subseteq {\bigcup}_{i=1}^6 \calS_i.
\]
Therefore, 
\begin{align} \label{E:delta bound}
1-\delta(k)=\frac{|\{u \in U(k): P_u(T) \text{ does not satisfy (\ref{E:Galois specific})}\}|}{|U(k)|} \leq \sum_{i=1}^6|\calS_i|/|U(k)|.
\end{align}

Now fix any $1\leq i \leq 6$.  Define $\bbar\rho_D$ as in \S\ref{S:big monodromy}.  We have
\[
\calS_i = \{ u \in U(k) : \bbar\rho_D(\Frob_u) \subseteq \calB_i \},
\]
where $\calB_i=\prod_{\ell \in D} (\Or(V_\ell) - C_i(V_\ell) )$.  Define  $G=\bbar\rho_D(\pi_1(U_k))$ and $G^g=\bbar\rho_D(\pi_1(U_{\bbar{k}}))$.    Note that $G^g$ is a normal subgroup of $G$ and $G/G^g$ is cyclic.  Let $hG^g$ be the $G^g$-coset of $G$ that contains $\bbar\rho_D(\Frob_u)$ for all $u\in U(k)$.  

\begin{lemma} \label{L:equidistribution 1}
We have
\[
\frac{|\calS_i|}{|U(k)|} =\frac{|\calB_i\cap hG^g|}{|G^g|} + O(|k|^{-1/2}),
\]
where the implicit constant does not depend on the choice of $k$.
\end{lemma}
\begin{proof}
Define the group $G_D^g:=\bbar\rho_D(\pi_1(U_{\bbar{F}}))$, where $F$ is the fraction field of $R$.  

We claim that the groups $G^g$ and $G_D^g$ are conjugate in $\prod_{\ell\in D} \Or(V_\ell)$.   The claim is easy if $R$ is a finite field since then $k$ is a finite extension of $F$ and the groups $G^g$ and $G_D^g$ depend only on the common algebraic closure of these fields.    The case where $R$ has characteristic $0$ follows from Proposition~\ref{P:big monodromy b}; recall that we have reduced to the case where the proposition holds with $S'=S$.

The lemma follows from an equidistribution result with enough control over the error terms, for example Theorem~9.7.13 of \cite{MR1659828}.  The above claim is needed to verify condition 9.7.2 (4) in \cite{MR1659828}.
\end{proof}

By (\ref{E:delta bound}) and Lemma~\ref{L:equidistribution 1}, we deduce that
\begin{align}\label{E:delta bound almost}
1-\delta(k)\leq \sum_{i=1}^6 \frac{|\calB_i\cap hG^g|}{|G^g|} + O(|k|^{-1/2}),
\end{align}
where the implicit constant does not depend on $k$.  

We now bound ${|\calB_i\cap hG^g|}/{|G^g|}$ for $1\leq i \leq 6$.  By Proposition~\ref{P:big monodromy a} and our choice of $\Lambda$, we have $G^g\supseteq {\prod}_{\ell\in D} \Omega(V_\ell)$.  Denote by $m$ the index of $ {\prod}_{\ell\in D} \Omega(V_\ell)$ in $G^g$. The $G^g$-coset $hG^g$ is the disjoint union of $m$ cosets of ${\prod}_{\ell\in D} \Omega(V_\ell)$; let $\kappa$ be any of these $m$ cosets.  We have $\kappa=\prod_{\ell\in D} \kappa_\ell$, where $\kappa_\ell$ is a $\Omega(V_\ell)$-coset in $\Or(V_\ell)$.  Therefore,
\[
\frac{|\calB_i \cap \kappa|}{|\kappa|} = \prod_{\ell\in D}\Big(1 - \frac{|C_i(V_\ell) \cap \kappa_\ell|}{|\kappa_\ell|} \Big) \leq (1-c_2/N^2)^{|D|}=(1-c_2/N^2)^{M} < \varepsilon/6,
\]
where the first inequality uses Proposition~\ref{P:criterion for sieving}(\ref{P:criterion for sieving b}) (note that $\ell\geq c_1$ for all $\ell\in \Sigma$) and the second inequality uses our choice of $M$.  Therefore,
\[
\frac{|\calB_i \cap hG^g|}{|G^g|} = \sum_{\kappa \subseteq hG^g} \frac{|\calB_i \cap \kappa|}{m |\kappa|} =\frac{1}{m} \sum_{\kappa \subseteq hG^g} \frac{|\calB_i \cap \kappa|}{ |\kappa|} < \varepsilon/6,
\]   
where the sums are over the $m$ cosets of ${\prod}_{\ell\in D} \Omega(V_\ell)$ contained in $hG^g$.    We deduce (\ref{E:delta goal}) from (\ref{E:delta bound almost}) and the above bound for ${|\calB_i \cap hG^g|}/{|G^g|}$.

\section{Proof of Theorem~\ref{T:main 2}}   \label{S:proof of main 2}

Fix notations and assumptions as in \S\ref{SS:new setup} and \S\ref{SS:effective version}.  Let $\Lambda$ be the set of natural density $1$ that satisfies condition (\ref{bm-b}) of \S\ref{SSS:bm}.  Let $c_1 \geq 5$ and $c_2$ be positive absolute constants as in Proposition~\ref{P:criterion for sieving}(\ref{P:criterion for sieving b}).    We may assume that each prime $\ell\in \Lambda$ is greater than $c_1$.

Take any $n\geq 1$.  After base extending everything to $\FF_{q^n}$, we find that the setup and assumptions of \S\ref{SS:new setup} and \S\ref{SS:effective version} still hold.  Moreover, we may take the same sets $\Sigma$ and $\Lambda$, and the integers $g$, $b$ and $N$ do not change.   So to prove  Theorem~\ref{T:main 2}, we may assume without loss of generality that $n=1$.   We may further assume that $U(\FF_q)$ is non-empty.\\

For each subset $D$ of $\Lambda$, define the representation
\[
\bbar\rho_D= \prod_{\ell \in D} \bbar\rho_\ell \colon \pi_1(U) \to \prod_{\ell\in D} \Or(V_\ell);
\]
note that the set $D$ may be infinite now.
Define the group $G_D:=\bbar{\rho}_D(\pi_1(U)) \subseteq \prod_{\ell\in D} \Or(V_\ell)$ and its normal subgroup $G_D^g:=\bbar{\rho}_D(\pi_1(U_{\FFbar_q}))$.  
We have $G_D^g \supseteq {\prod}_{\ell\in D} \Omega(V_\ell)$; this follows for finite $D$ by Proposition~\ref{P:big monodromy a} and hence infinite $D$ since the groups involved are profinite.   

Denote the index of ${\prod}_{\ell\in \Lambda} \Omega(V_\ell)$ in $G_\Lambda^g$ by $m$.

\begin{lemma} \label{L:finite abelianization}
The value $m$ is finite and satisfies $m\leq 2^{2g+b-1}$.   We have $[G_\Lambda: G_\Lambda^g] \leq 2$.
\end{lemma}
\begin{proof}
Since the groups ${\prod}_{\ell\in \Lambda} \Omega(V_\ell)$ and $G_\Lambda^g$ are profinite, to bound $m$ it suffices to prove that 
\[
[G_D^g:{\prod}_{\ell\in D} \Omega(V_\ell)]\leq 2^{2g+b-1}
\]
 for any fixed finite $D\subseteq \Lambda$.  Define $H= G_D^g/\prod_{\ell\in D} \Omega(V_\ell)$; it is a subgroup of $(\prod_{\ell \in D} \Or(V_\ell))/(\prod_{\ell \in D} \Omega(V_\ell) )\cong (\ZZ/2\ZZ)^{2|D|}$.   Therefore, $H$ is isomorphic to $(\ZZ/2\ZZ)^r$ for some integer $r$.  Let $G$ be a finite group with cardinality relatively prime to $q$ that is a quotient of $\pi_1(U_{\FFbar_q})$.   Corollaire~2.12 of \cite{MR2017446}*{XIII} implies that $G$ can be generated by $2g+b-1$ elements.   Since $q$ is odd, we deduce that the group $H$ is generated by $2g+b-1$ elements.   Therefore, $r \leq 2g+b-1$ and hence $|H|\leq 2^{2g+b-1}$.

The group $G_\Lambda/G_\Lambda^g$ is pro-cyclic since it is a quotient of the absolute Galois group of $\FF_q$.   However, every element in $G_\Lambda/G_\Lambda^g$ has order $1$ or $2$ since it is a quotient of  
\[
G_\Lambda/(\prod_{\ell \in \Lambda} \Omega(V_\ell)) \subseteq (\prod_{\ell \in \Lambda} \Or(V_\ell))/(\prod_{\ell \in \Lambda} \Omega(V_\ell)) \cong \prod_{\ell\in \Lambda} (\ZZ/2\ZZ)^2.
\]  
Therefore, $G_\Lambda/G_\Lambda^g$ is cyclic of order $1$ or $2$.
\end{proof}

Let $hG_\Lambda^g$ be the coset of $G_\Lambda^g$ in $G_\Lambda$ which contains $\bbar\rho_\Lambda(\Frob_u)$ for all $u\in U(\FF_{q})$.     Fix one of the $m$ cosets $\kappa$ of $\prod_{\ell\in \Lambda} \Omega(V_\ell)$ in $G_\Lambda$ that is also a subset of $hG_\Lambda^g$.   We have $\kappa=\prod_{\ell\in \Lambda} \kappa_\ell$ for unique cosets $\kappa_\ell$ of $\Omega(V_\ell)$ in $\Or(V_\ell)$.   We also fix an integer $1\leq i\leq 6$. 
\\

 With $\kappa$ and $i$ fixed, let $A$ be the set of $u\in U(\FF_{q})$ that satisfy $\bbar\rho_\Lambda(\Frob_u)\subseteq \kappa$.  Let $\{C_i(V_\ell)\}_{\ell\in \Lambda}$ be the sets from Proposition~\ref{P:criterion for sieving}. 
  For a prime $\ell\in \Lambda$, let $A_\ell$ be the set of $u\in A$ for which $\bbar\rho_\ell(\Frob_u)\subseteq C_i(V_\ell)$ and define $ \omega_\ell:={|C_i(V_\ell)\cap \kappa_\ell|}/{|\kappa_\ell|}$.   For a subset $D\subseteq \Lambda$, define $A_D=\cap_{\ell\in D} A_\ell$; it is the set of $u\in A$ satisfying $\bbar\rho_\ell(\Frob_u)\subseteq C_i(V_\ell)$ for all $\ell\in D$.   

\begin{lemma} \label{L:hard equi}
 For every finite subset $D\subseteq \Lambda$, we have
\[
|A_D| =   \frac{|U(\FF_{q})|}{m} \cdot \prod_{\ell\in D} \omega_\ell + r_D,
\] 
where $|r_D|\leq (\prod_{\ell\in D} \ell)^{N(N-1)/4}(2g+b) q^{1/2}$.
\end{lemma}
\begin{proof}
Take any finite subset $D\subseteq \Lambda$.   Since $m$ is finite by Lemma~\ref{L:finite abelianization}, there is a non-empty finite set $D\subseteq E \subseteq \Lambda$ such that the projection map 
\[
G_\Lambda/{\prod}_{\ell\in \Lambda} \Omega(V_\ell)\to G_E/{\prod}_{\ell\in E} \Omega(V_\ell)
\] 
is an isomorphism.  In particular, for $u\in U(\FF_{q})$, we have $\bbar\rho_\Lambda(\Frob_u) \subseteq \kappa$ if and only if $\bbar\rho_{E}(\Frob_u)\subseteq \prod_{\ell\in E} \kappa_\ell$.

Define 
\[
B:= \prod_{\ell\in D} (C_i(V_\ell)\cap \kappa_\ell) \times \prod_{\ell\in E-D} \kappa_\ell;
\] 
it is a subset of $G_{E}$ that is stable under conjugation.   Observe that 
\[
A_D = \{ u\in U(\FF_q) : \bbar\rho_{E}(\Frob_u)\subseteq B\}.
\]   

Define the subgroup $H:=\prod_{\ell \in D} \{I\} \times \prod_{\ell\in E-D} \Omega(V_\ell)$ of $G_{E}^g$; it is a normal subgroup of $G_{E}$ and satisfies  $B\cdot H \subseteq B$.  
The representation $\bbar \rho_E$ is tamely ramified since the representations $\{\bbar\rho_\ell\}_{\ell \in \Lambda}$ are tamely ramified by assumption.  By Theorem~\ref{T:equidistribution}(\ref{T:equidistribution ii}) in Appendix~\ref{S:B}, we have
\[
|A_D| = \frac{|B|}{|G_E^g|}\cdot |U(\FF_{q})| + r_D,
\] 
where $r_D$ satisfies $|r_D|\leq |B|^{1/2}/|H|^{1/2} \cdot (2g+b) q^{1/2}$.  By our choice of $E$, the index $[G_{E}^g:{\prod}_{\ell\in E} \Omega(V_\ell)]$ equals $m$.  Therefore,
\[
\frac{|B|}{|G_{E}^g|} = \frac{1}{m} \prod_{\ell\in D} \frac{|C_i(V_\ell) \cap \kappa_\ell|}{|\Omega(V_\ell)|} = \frac{1}{m} \prod_{\ell\in D} \omega_\ell
\]
and it thus remains to prove the correct bound for $|r_D|$.  We have $|B|/|H| \leq ({\prod}_{\ell \in E} |\Omega(V_\ell)|)/|H| =  \prod_{\ell \in D} |\Omega(V_\ell)|$ and hence $|r_D|\leq \prod_{\ell \in D} |\Omega(V_\ell)|^{1/2} \cdot (2g+b) q^{1/2}$.
It thus remains to prove that $|\Omega(V_\ell)|\leq  \ell^{N(N-1)/2}$ for all $\ell\in D$.

Take any $\ell \in D$.   The possible cardinality for $|\Or(V_\ell)|$ is given in \cite{MR2562037}*{\S3.7.2}.  If $N=2n+1$ is odd,  we find that $|\Or(V_\ell)|\leq 2 \ell^{m^2+2+4+\cdots +2m}= 2\ell^{N(N-1)/2}$.   If $N=2n$ is even,  we find that $|\Or(V_\ell)|\leq 2 \ell^{m(m-1) +(2+4+\cdots +2(m-1)) +m}= 2\ell^{N(N-1)/2}$.    Therefore, $|\Omega(V_\ell)| \leq \ell^{N(N-1)/2}/2$.
\end{proof}

We will now use Selberg's sieve, as described in Appendix~\ref{S:The Selberg sieve}, to bound the cardinality of the set 
\[
S_{\kappa,i} :=\{ u \in A : \bbar\rho_{\ell}(\Frob_u) \not\subseteq C_i(V_\ell) \text{ for all }\ell\in \Lambda\}.
\]

\begin{lemma} \label{L:main sieve}
We have
\[
|S_{\kappa,i}| \ll \Big(m^{-1} |U(\FF_q)|  \log q +  (2g+b) q \Big) q^{-1/(N^2-N+6)},
\] 
where the implicit constant depends only on $\Lambda$.
\end{lemma}
\begin{proof}
For each $Q\geq 1$, let $\Lambda(Q)$ be the set of primes $\ell\in \Lambda$ with $\ell\leq Q$.  Since $\Lambda$ has positive natural density, there is a constant $c_3\geq 1$ such that 
\[
|\Lambda(Q)| \gg Q/\log Q
\] 
for all $Q\geq c_3 $, where $c_3$ and the implicit constant depend only on $\Lambda$.  

Set $X:= |U(\FF_{q})|/m$.  For each finite $D\subseteq \Lambda$, we  have $|A_D| = (\prod_{\ell \in D} \omega_\ell) X + r_D$, where $r_D$ satisfies the inequality from Lemma~\ref{L:hard equi}.   We may assume that $\omega_\ell< 1$ for all $\ell \in \Lambda$ since otherwise $S_{\kappa,i}=\emptyset$ and the desired upper bound is trivial.   We have $\ell \geq c_1$, and hence $\omega_\ell \geq c_2/N^2$, for all $\ell\in \Lambda$.  In particular, $\omega_\ell >0$ for all $\ell\in \Lambda$.

Fix a number $Q\geq c_3$.  Observe that $S_{\kappa,i}$ is a subset of $A-(\cup_{\ell\in \Lambda(Q)} A_\ell)$. Let $\Zz(Q)$ be the set of finite subsets $D$ of $\Lambda$, equivalently of $\Lambda(Q)$, such that ${\prod}_{\ell \in D}  \ell \leq Q$.  We have $|\Zz(Q)|\leq Q$. Therefore,
\begin{align*}
\sum_{D,D'\in\Zz(Q)} |r_{D\cup D'}| &\leq |\Zz(Q)|^2\cdot (Q^2)^{N(N-1)/4}  (2g+b) q^{1/2}  \leq  Q^{N(N-1)/2+2}  (2g+b) q^{1/2}.
\end{align*}
By the \defi{Selberg sieve} (Theorem~\ref{T:Selberg sieve}), we obtain  the bound
\[
|S_{\kappa,i}| \leq X/H(Q) +  Q^{N(N-1)/2+2}  (2g+b) q^{1/2},
\]
where $H(Q):=\sum_{D\in\Zz(Q)}\prod_{\ell\in D} \omega_\ell/(1-\omega_\ell)$.  Since $Q\geq c_3$, we have
\[
H(Q) \geq \sum_{\ell\in \Lambda(Q) } \omega_\ell \geq \frac{c_2}{N^2}\cdot |\Lambda(Q)| \gg \frac{1}{N^2} Q/\log Q,
\]
where we have used Proposition~\ref{P:criterion for sieving}(\ref{P:criterion for sieving b}).  Therefore,
\[
|S_{\kappa,i}| \ll m^{-1} |U(\FF_q)|\cdot N^2 \log(Q)/Q + Q^{N(N-1)/2+2} (2g+b) q^{1/2}.
\]

Set $Q:= 	q^{1/(N^2-N+6)}$.  If $Q\geq c_3$, then
\begin{align} \label{E:Skappai}
|S_{\kappa,i}| \ll \Big(m^{-1} |U(\FF_q)|  \log q +  (2g+b) q \Big) q^{-1/(N^2-N+6)}.
\end{align}
If $Q < c_3$, then the bound (\ref{E:Skappai}) is immediate since 
\[
(2g+b) q \cdot q^{-1/(N^2-N+6)} \gg (2g+b) q \gg q+2g\sqrt{q} +1 \geq |U(\FF_q)| \geq |S_{\kappa,i}|. \qedhere
\]
\end{proof}

Since $hG_\Lambda^g$ is the union of $m$ cosets $\kappa_1,\ldots, \kappa_m$ of $\prod_{\ell\in \Lambda} \Omega(V_\ell)$, we have
\begin{align*} 
&{|\{u\in U(\FF_q): \bbar\rho_\ell(\Frob_u)\not\subseteq C_i(V_\ell) \text{ for all }\ell\in\Lambda\}|} \\
 \leq  &\sum_{j=1}^m {|S_{\kappa_j,i}|}  \ll  \Big( |U(\FF_q)|  \log q +  m (2g+b) q \Big) q^{-1/(N^2-N+6)},
\end{align*}
where the last inquality uses Lemma~\ref{L:main sieve}.  By Proposition~\ref{P:criterion for sieving}(\ref{P:criterion for sieving c}) and Lemma~\ref{L:finite abelianization}, we deduce that 
\begin{align*}
1-\delta(\FF_q) & = \frac{|\{u \in U(\FF_q): P_u(T) \text{ does not satisfies (\ref{E:Galois specific})}\} |}{|U(\FF_q)|} \\
&  \ll  \big(  \log q +  2^{2g+b} (2g+b) q/|U(\FF_q)| \Big) q^{-1/(N^2-N+6)}.
\end{align*}
 
If $g \leq \sqrt{q}/4$ and $b\leq q/4$, then $|U(\FF_q)| \geq q+1-2g\sqrt{q} -b\geq q/4$ and hence
\begin{align*} 
1-\delta(\FF_q) \ll \big(  \log q +  2^{2g+b} (2g+b)  \big) q^{-1/(N^2-N+6)} \ll  2^{2g+b} (2g+b) \,q^{-1/(N^2-N+6)} \log q.
\end{align*}

Finally suppose that $g \geq \sqrt{q}/4$ or $b\geq q/4$.  Using $N\geq 3$, we find that 
\[
2^{2g+b} (2g+b) q^{-1/(N^2-N+6)} \log q    \geq  2^{2g+b}  q^{-1/12} \geq 2^{\sqrt{q}/4} q^{-1/12} \gg 1 \geq 1-\delta(\FF_q).
\]

\section{Proof of Theorem~\ref{T:hypersurfaces}}
\label{S:hypersurfaces}

Let $\PP$ be the projective space over $\ZZ$ consisting of non-zero homogenous polynomials of degree $d$ in variables $x_0,\ldots, x_{n+1}$ up to scalars.  By ordering the monomials in $x_0,\ldots, x_{n+1}$ of degree $d$, we obtain an isomorphism $\PP\cong \PP^{m}_\ZZ$ where $m= {n+1+d \choose d}-1$.  Let $U \subseteq \PP$ be the open subscheme corresponding to homogeneous polynomials that define a \emph{smooth} hypersurface. From \S11.4.7 of \cite{MR1659828}, we know that $U$ is smooth, connected, and that $U(k)$ is nonempty for all fields $k$.  Let $H \subseteq U \times \PP^{n+1}$ be the subscheme defined by pairs consisting of a homogeneous polynomial and a point on the corresponding hypersurface.  The projection
\[
\pi\colon H \to U
\]
gives the \emph{universal family of degree $d$ hypersurfaces} in $\PP^{n+1}$.  For each point $f\in U(k)$, with $k$ a field, we denote by $H_f$ the fiber of $\pi$ over $f$.   Note that $H_f$ is the hypersurface of $\PP^{n+1}_k$ corresponding to $f$ and agrees with the notation introduced in \S\ref{SS:hypersurface}.    \\

We now show that the setup of \S\ref{SS:new setup} applies with $R=\ZZ$.    The following simply summarizes material presented by Katz in \S8 of \cite{katzreport} with $X=\PP^{n+1}_\ZZ$.  Take a prime $\ell\geq 5$.  We have a lisse $\ZZ_\ell$-sheaf  $R^n \pi_* \ZZ_\ell(n)$ on $U_{\ZZ[1/\ell]}$.   The cup product 
\[
R^n \pi_* \ZZ_\ell(n) \times R^n \pi_* \ZZ_\ell(n) \to R^{2n} \pi_* \ZZ_\ell(2n) \cong \ZZ_\ell
\]
is an orthogonal autoduality modulo torsion (that the pairing is symmetric uses that $n$ is even).   On $\Spec \ZZ[1/\ell]$ we have the lisse $\ZZ_\ell$-sheaf $R^n\gamma_* \ZZ_\ell(n)$,  where $\gamma\colon \PP^{n+1}_{\ZZ[1/\ell]} \to \Spec \ZZ[1/\ell]$ is the structure morphism.    The sheaf $R^n\gamma_* \ZZ_\ell(n)$ pulls back to a sheaf $\calF_\ell$ on $U_{\ZZ[1/\ell]}$.   We can view $\calF_\ell$ as a subsheaf of $R^n \pi_* \ZZ_\ell(n)$, and we define $\Ev_{\ZZ_\ell}$ to be the orthogonal to $\calF_\ell$ under the cup product pairing.  

For $\ell$ sufficiently large, the lisse sheaf $\Ev_{\ZZ_\ell}$ is torsion free and the cup product makes $\Ev_{\ZZ_\ell}$ self dual over $\ZZ_\ell$.   With such $\ell$, let $M_\ell$ be the fiber of $\Ev_{\ZZ_\ell}$ at a geometric fiber of $U$; it gives rise to a representation
\[
\rho_\ell\colon \pi_1(U_{\ZZ[1/\ell]}) \to \Or(M_\ell)
\]
These representations $\rho_\ell$ are compatible and the corresponding polynomials $P_f(T)$ are those described in \S\ref{SS:hypersurface}.  Note that the description of the zeta function of $H_f$ from  \S\ref{SS:hypersurface} is given in the second half of \S8 of \cite{katzreport}.   The zeta functions are also described in \S11.4 of \cite{MR1659828} where it is observed that their common degree is $N :=(d-1)((d-1)^{n+1}+1)/d$.  So the $M_\ell$ have common rank $M$ over $\ZZ_\ell$ and $N>2$.

 In \S8 of \cite{katzreport}, Katz observes that the representations $\rho_\ell$ satisfy condition (\ref{bm-a}) in \S\ref{SSS:bm}.   Moreover, he notes that the Zariski closure in condition (\ref{bm-a}) is always the full group $\Or_{\calV_\ell}$; using this and equidistribution, one can prove Remark~\ref{R:sign of hyper}.  For this big monodromy result, we need our assumptions $d\geq 3$ and $(n,d)\neq (2,3)$.
\\

Using $N=(d-1)((d-1)^{n+1}-1)/d$ and $n$ even, we find that $N$ is even if and only if $d$ is odd.   The following, which we will prove in \S\ref{SS:K hypersurface proof}, describes the field $K$ from \S\ref{SS:the field K} when $N$ is even.

\begin{lemma} \label{L:K hypersurface}
Suppose that $N$ is even (equivalently, $d$ is odd).  Then $K=\QQ(\sqrt{(-1)^{(d-1)/2} d})$.  Moreover, $K=\QQ$ if and only if $d$ is a square.
\end{lemma}

We have verified the axiomatic setup of \S\ref{SS:new setup}.  Lemma~\ref{L:K hypersurface} describes the field $K$ when $N$ is even and in particular describes when $K=\QQ$.  Theorem~\ref{T:hypersurfaces} now follows from Theorem~\ref{T:main 1}.

\subsection{Proof of Lemma~\ref{L:K hypersurface}} \label{SS:K hypersurface proof}

 Let $\calX$ be a smooth hypersurface of degree $d$ in $\PP^{n+1}_\CC$ and define the complex manifold $X:=\calX(\CC)$. Let $h$ in $H^n(X,\ZZ)$ be the class of a linear section of codimension $n/2$; we have $h^2=d$.
Let $L:=H^n(X,\ZZ)_\circ$ be the \emph{primitive cohomology lattice}, i.e., the orthogonal complement in $H^n(X,\ZZ)$ of the class $h$ with respect to the usual intersection pairing.  Note that $L$ is a lattice, i.e., an orthogonal space over $\ZZ$, and so the discriminant of $L$ is a well-defined integer.   Beauville \cite{MR2588789}*{Theorem~4} describes the structure of $L$ from which it is clear that $\disc(L)=\pm d$.

We can take $M_\ell$ to be the fiber of the sheaf $\Ev_{\ZZ_\ell}$ above the complex point corresponding to $\calX$.   For $\ell$ sufficiently large, the orthogonal space $M_\ell$ will be isomorphic to $L\otimes_\ZZ \ZZ_\ell$.   So for $\ell$ sufficiently large, the orthogonal space $V_\ell:=M_\ell/\ell M_\ell$ over $\FF_\ell$ will have discriminant $\disc(L)\cdot (\FF_\ell^\times)^2$.   

From the description of $K$ in \S\ref{SS:the field K}, a sufficiently large prime $\ell$ splits in $K$ if and only if $(-1)^{N/2} \disc(L)$ is a square modulo $\ell$.  Therefore, $K=\QQ(\sqrt{(-1)^{N/2} \disc(L)})$.   Using that $N=(d-1)((d-1)^{n+1}-1)/d$ and $d$ is odd, we find that $N\equiv (d-1)(-1)/d \equiv d-1 \pmod{4}$.   Therefore, $K=\QQ(\sqrt{(-1)^{(d-1)/2} \disc(L)})$.

We will show that $\disc(L)=d$ and hence $K=\QQ(\sqrt{(-1)^{(d-1)/2} d})$.   For $K$ to be $\QQ$, we certainly need $d$ to be a square.  If $d$ is a square, then $d\equiv 1\pmod{4}$ since it is odd and thus $K=\QQ$.\\

It remains to prove that $\disc(L)=d$.  Since $\disc(L)=\pm d$, we need only show that $\disc(L)$ is positive.

We now consider the cohomology group $H^n(X,\RR)$.  The cup product gives a non-degenerate symmetric pairing $H^n(X,\RR) \times H^n(X,\RR) \to \RR$.  So $H^n(X,\RR)$ is an orthogonal space over $\RR$ and we will now compute its discriminant; there are two possibilities $(\RR^\times)^2$ and $-1\cdot (\RR^\times)^2$.  We claim that $\disc(H^n(X,\RR))=(\RR^\times)^2$.   Since $H^n(X,\RR) = L\otimes_\ZZ \RR \oplus \RR h$ and $h^2=d>0$, this  claim will prove that $\disc(L)$ is positive.
There is an orthogonal basis $v_1,\ldots, v_m$ over $\RR$ of $H^n(X,\RR)$.  By scaling the vectors, we may assume that $\ang{v_i}{v_i} = \pm 1$.    Let $b^+$ and $b^-$ be the number of $v_i$ for which $\ang{v_i}{v_i}$ is $1$ and $-1$, respectively.   The discriminant of $H^n(X,\RR)$ is thus equal to $(-1)^{b_-}(\RR^\times)^2$; so to complete the proof of Lemma~\ref{L:K hypersurface}, it suffices to show that $b_-$ is even.

\begin{lemma} \label{L:hodge index needed}
We have $b_+ - b_- \equiv d \pmod{4}$.
\end{lemma}
\begin{proof}
The Hodge index theorem \cite{MR1967689}*{Theorem~6.33} shows that
\[
b_+ - b_- = \sum_{p,q} (-1)^p h^{p,q}(X),
\] 
where $h^{p,q}(X)$ is the $(p,q)$-Hodge number of $X$.    For $0\leq i \leq 2n$ with $i\neq n$, $\dim_\RR H^i(X,\RR)$ is $0$ if $i$ is odd and $1$ if $i$ is even, cf.~\cite{MR1659828}*{\S11.4.2}.  So when $p+q\neq n$, we have $h^{p,q}(X)=1$ if $0\leq p=q \leq n$ and $h^{p,q}(X)=0$ otherwise.  Therefore, 
\[
b_+ - b_- = \sum_{p+q=n} (-1)^p h^{p,q}(X) + \sum_{0\leq i \leq n,\, i\neq n/2} (-1)^{i} 
= \sum_{p+q=n} (-1)^p h^{p,q}(X) +1-(-1)^{n/2}.
\]

By Hirzebruch's formula for Hodge numbers, cf.~Th\'eor\`eme~2.3 of Expos\'e XI of \cite{MR0354657}, we have the following equality
\[
\sum_{p\geq 0 ,q\geq 0 } h^{p,q}_\circ \,y^p z^q = \frac{1}{(1+y)(1+z)} \Big( \frac{(1+y)^d -(1+z)^d}{-(1+y)^d z +(1+z)^d y}  -1 \Big)
\]
in $\ZZ[\![y,z]\!]$, where $h^{p,q}_\circ:=h^{p,q}-\delta_{p,q}$ and $h^{p,q}$ is the $(p,q)$-Hodge number of any smooth hypersurface of degree $d$ in $\PP^{2(p+q)+1}_\CC$.  Setting $y=-x$ and $z=x$, we have
\[
\sum_{m\geq 0} \Big(\sum_{p+q=m} (-1)^p h^{p,q}_0\Big) \, x^m= \frac{1}{(1-x)(1+x)} \Big( \frac{(1-x)^d -(1+x)^d}{-(1-x)^d x -(1+x)^d x}  -1 \Big) = \frac{1}{1-x^2}( \alpha/\beta -1),
\]
where $\alpha:=-\big((1-x)^d -(1+x)^d\big)/(2x)$ and $\beta:= ((1-x)^d+(1+x)^d)/2$.   Expanding out $\alpha$ and $\beta$, we find that
\begin{align*}
\alpha= & -\tfrac{1}{2} {\sum}_{i \geq 0} \tbinom{d}{i} ((-1)^i-1) x^{i-1} = {\sum}_{k\geq 0} \tbinom{d}{2k+1} x^{2k} \quad \text{ and} \\
\beta=&  \tfrac{1}{2} {\sum}_{i \geq 0} \tbinom{d}{i} ((-1)^i+1) x^{i}={\sum}_{k\geq 0} \tbinom{d}{2k}  x^{2k}.
\end{align*}
In particular, we have $\alpha,\beta \in \ZZ[\![x]\!]$.  For each $k\geq 0$, we have
\[
{d \choose 2k+1} - d {d \choose 2k} = {d \choose 2k} \Big(\frac{d-2k}{2k+1} - d \Big)= {d \choose 2k} \cdot \frac{-2k(d+1)}{2k+1} \equiv 0 \pmod{4},
\]
where the congruence uses that $d$ is odd.  Therefore, $\alpha\equiv d \beta \pmod{4}$.  The constant term of $\beta$ is $1$, so $\beta^{-1} \in \ZZ[\![x]\!]$ and hence $\alpha/\beta \equiv d \pmod{4}$.    So
\[
\sum_{m\geq 0} \Big(\sum_{p+q=m} (-1)^p h^{p,q}_0\Big) \, x^m \equiv \frac{1}{1-x^2} (d-1) =(d-1)(1+x^2+x^4+x^6+\cdots) \pmod{4}
\]
and hence
\[
\sum_{p+q=n} (-1)^p h^{p,q} = \sum_{p+q=n} (-1)^p h_\circ^{p,q} + (-1)^{n/2} \equiv d-1 + (-1)^{n/2} \pmod{4}.
\]
Therefore, $b_+ - b_- \equiv (d-1 + (-1)^{n/2})  +1-(-1)^{n/2} \equiv d  \pmod{4}$.
\end{proof}

We have  $b_+ + b_- = N+1 = (d-1)((d-1)^{n+1}+1)/d + 1$.  Using that $d$ is odd and $n+1\geq 2$, we find that $b_+ + b_- \equiv (d-1)/d+1 \equiv d \pmod{4}$.   By Lemma~\ref{L:hodge index needed}. we deduce that 
\[
2b_- = (b_++b_-)-(b_+-b_-) \equiv  d -d =0 \pmod{4}.
\]
This implies that $b_-$ is even as desired.

\section{Proof of Theorems~\ref{T:main EC} and \ref{T:main 2 EC}  } \label{S:cohomological interpretation}
We first check the axiomatic setup of \S\ref{SS:new setup} with $R=\FF_q$ and $U=U_d$.  Let $\Sigma$ be the set of primes $\ell\geq 5$ that do not divide $q$.  \\ 

Take any $\ell\in \Sigma$.  Following Katz, Hall constructs in \S6.2 of \cite{MR2372151} a representation
\[
\bbar\rho_\ell\colon \pi_1(U_d)\to \Or(V_\ell),
\]
with $V_\ell$ an orthogonal space over $\FF_\ell$, satisfying
\[
P_u(T) \equiv \det(I-\bbar\rho_\ell(\Frob_u)T) \pmod{\ell}
\]
for all $n\geq 1$ and $u\in U_d(\FF_{q^n})$.   One can easily see that $\bbar\rho_\ell$ arises from a representation $\rho_\ell\colon \pi_1(U_d)\to \Or(M_\ell)$, with $M_\ell$ an orthogonal space over $\ZZ_\ell$ and $V_\ell\cong M_\ell/\ell M_\ell$, satisfying 
\begin{align} \label{E:reversed}
P_u(T) = \det(I-\rho_\ell(\Frob_u)T) 
\end{align}
for all $n\geq 1$ and $u\in U_d(\FF_{q^n})$ (in Hall's construction, simply replace $\mathcal{T}_{d,\ell}$ with the $\ZZ_\ell$-sheaf $\mathcal{T}_{d,\ell^\infty}$ described in \S6.6 of \cite{MR2372151}).  The common dimension of the $V_\ell$ is our integer $N_d$ by \cite{MR2372151}*{Lemma~6.2}.  We have $N_d\geq 3$ since by assumption.

It remains to verify that condition (\ref{bm-b}) in \S\ref{SSS:bm} holds.  To do this, we will restrict to a subvariety of $U_d$; after possibly replacing $\FF_q$ by a finite extension, one can further assume that $U_{d-1}(\FF_q)$ is non-empty.\\

Now fix a polynomial $g \in   U_{d-1}(\FF_q)$.   We let $U$ be the subvariety of $\mathbb{A}^1_{\FF_q}$ consisting of $c$ for which $(t-c)g(t)$ is separable and relatively prime to $m(t)$.   We can identify $U$ with a closed subvariety of $U_d$ via the map $c\mapsto (t-c)g(t)$.   Restricting $\rho_\ell$ and $\bbar\rho_\ell $ to $\pi_1(U)$ gives representations $\varrho_\ell\colon \pi_1(U)\to\Or(M_\ell)$ and $\bbar\varrho_\ell\colon  \pi_1(U)\to\Or(V_\ell)$.  These representations satisfy the axiomatic setup of \S\ref{SSS:rep} and \S\ref{SSS:compatibility} with $R=\FF_q$ and the same set $\Sigma$ from the above discussion.  Moreover, each representation $\varrho_\ell$ is tamely ramified, cf.~\cite{MR2372151}*{\S6.3}.\\

  Let $\Lambda$ be the set of $\ell\in \Sigma$ which do not divide $\max\{1,-\ord_v(j_E)\}$ for any place $v$ of $\FF_q(t)$, where $j_E\in \FF_q(t)$ is the $j$-invariant of $E$.  We now show that condition (\ref{bm-b}) holds for the representations $\{\varrho_\ell\}_{\ell \in \Gamma}$. 

\begin{lemma} \label{L:big monodromy ec}
For each prime $\ell$, we have $\bbar\varrho_\ell(\pi_1(U_{\FFbar_q}))\supseteq \Omega(V_\ell)$ and $\bbar\varrho_\ell(\pi_1(U_{\FFbar_q}))$ is not a subgroup of $\SO(V_\ell)$.
\end{lemma}
\begin{proof}
After replacing $E$ by its quadratic twist by $g(t)$, we may assume without loss of generality that $d=1$.  Note that performing this twist leaves the integer $B$ unchanged.  Using the assumptions of the theorems, there will be a place $v\neq \infty$ of $\FF_q(t)$ for which $E$ has Kodaira symbol $\operatorname{I}_0^*$.  There is also a place $v\neq \infty$ for which $E$ has multiplicative reduction, i.e., $E$ has Kodaira symbol $\operatorname{I}_n$ at $v$ for some $n\geq 1$.  The lemma is now a direct consequence of Theorem~3.4 of \cite{orthogonal} which is an explicit version of Theorem~6.4 of \cite{MR2372151}.
\end{proof}

An immediate consequence of Lemma~\ref{L:big monodromy ec} is that $\bbar\rho_\ell(\pi_1(U_{d,\FFbar_q}))\supseteq \Omega(V_\ell)$ for all $\ell\in \Lambda$ and $\bbar\rho_\ell(\pi_1(U_{d,\FFbar_q}))$ is not a subgroup of $\SO(V_\ell)$.  

\begin{remark}
Using that $\bbar\rho_\ell(\pi_1(U_{d,\FFbar_q}))$ is not a subgroup of $\SO(V_\ell)$ and equidistribution, one can prove Remark~\ref{R:ECs}(\ref{R:ECs eps}) which says that  $|\{u\in U_d(\FF_{q^n}): \varepsilon_u=\varepsilon\}|/|U(\FF_{q^n})| \to 1/2$ as $n\to \infty$ for each $\varepsilon \in \{\pm 1\}$
\end{remark}

We have now verified enough to apply Theorems~\ref{T:main 1} and \ref{T:main 2} to the representations $\{\rho_\ell\}_{\ell\in \Lambda}$.    Note that $U$ is open in $\mathbb{A}^1_{\FF_q}\subseteq \PP^1_{\FF_q}$ and $|(\PP^1-U)(\FFbar_q)|=d+\deg m$.\\

Theorems~\ref{T:main EC} and \ref{T:main 2 EC} are now immediate if we can prove that $K=\QQ(\sqrt{(-1)^{N_d/2} D_d})$ if $N_d$ is even.

Now suppose that $N_d$ is even.  It remains to compute the field $K$ from \S\ref{SS:the field K} and determine when $K=\QQ$.  The following lemma depends on a result from \cite{orthogonal} which uses known cases of the Birch and Swinnerton-Dyer conjecture for elliptic curves over global function fields.

\begin{lemma} \label{L:Dd class}
For $\ell\in \Lambda$, we have $\disc(V_\ell)=D_d \cdot (\FF_\ell^\times)^2$.
\end{lemma}
\begin{proof}
Take any $\ell\in \Lambda$.  By Lemma~\ref{L:big monodromy ec}, there is an element $g\in \bbar\varrho_\ell(\pi_1(U))$ such that $\det(I\pm g)\neq 0$.   By equidistribution, there is some $c\in U(\FF_{q^n})$ such that $\bbar\varrho_\ell(\Frob_c)$ is conjugate to $g$ in $\Or(V_\ell)$.  By \cite{orthogonal}*{Proposition~3.2(e)}, we have $\disc(V_\ell)=D \cdot (\FF_\ell^\times)^2$,
where $D:=\prod_{v} \gamma_v(E_{t-c})^{\deg v}$ and the product is over places $v$ of $\FF_{q^n}(t)$.  We have $D=\gamma_\infty(E_{t^d}) \prod_{v\neq \infty} \gamma_v(E_{t-c})^{\deg v}$, where the product is over places $v$ of $\FF_{q^n}(t)$.  We have $D=D_d$ by noting that the integer $\gamma_\infty(E_{t^d}) \prod_{v\neq \infty} \gamma_v(E_{t-c})^{\deg v}$ does not change if we consider $v$ running over places of $\FF_q(t)$ instead of $\FF_{q^n}(t)$.
\end{proof}

By Lemma~\ref{L:Dd class}, we have $K=\QQ(\sqrt{(-1)^{N_d/2} D_d})$.   In particular,  $K=\QQ$ if and only if $(-1)^{N_d/2} D_d$ is a square.

\appendix

\section{The Selberg sieve}
\label{S:The Selberg sieve}

In this appendix, we give a version of Selberg's sieve.   This elegant and useful method was introduced by Selberg in \cite{MR0022871} to sieve integers by congruences modulo primes.  For background, see \cite{MR2061214}*{\S6.5} or \cite{MR2200366}*{\S7.2}.  For future reference, we give a version that is more general than what is required for our application.

\begin{theorem} \label{T:Selberg sieve} 
Let $A$ be a measure space with a bounded measure $\mu$.  Let $\Lambda$ be a finite set, and for each $\lambda\in\Lambda$ fix a measurable subset $A_\lambda$ of $A$.  Define the set
\[
S: = A - \big(\cup_{\lambda\in\Lambda} A_\lambda\big).
\]
Fix real numbers $\{\omega_\lambda\}_{\lambda\in\Lambda}$ with $0< \omega_\lambda<1$ and $X\geq 0$.  Define $A_D:=\cap_{\lambda\in D} A_\lambda$ for each non-empty $D\subseteq\Lambda$ and set $A_\emptyset:=A$.  
Let $r_D$ be the real number satisfying
\begin{equation} \label{E:remainder}
\mu(A_D)=\Big(\prod_{\lambda\in D} \omega_\lambda \Big)\cdot X + r_D.
\end{equation}

Let $\Zz$ be a set of subsets of $\Lambda$ such that if $D\in \Zz$ and $E\subseteq D$, then $E\in \Zz$.   Then
\begin{equation}\label{E:main bound}
\mu(S) \leq \frac{X}{H} + \sum_{D,D'\in\Zz} |r_{D\cup D'}|
\end{equation}
where $\displaystyle H:= \sum_{D\in\Zz} \prod_{\lambda\in D} \frac{\omega_\lambda}{1-\omega_\lambda}$.  (When $H=0$, we interpret this as giving the trivial bound $\mu(S)\leq +\infty$.)
\end{theorem}

Before proceeding, let us first give some context. After normalizing the measure, we may assume that $(A,\mu)$ is a probability space and hence use the language of probability.   For each $\lambda\in \Lambda$, we have fixed an event $A_\lambda$.  So $S$ is the set of outcomes that do not belong to any of the elements $A_\lambda$.

Consider the special case where the events $\{A_\lambda\}_{\lambda\in \Lambda}$ are independent.   We have $\mu(S)=\prod_{\lambda\in \Lambda} (1-\omega_\lambda)$.  Set $\omega_\lambda=\mu(A_\lambda)$ and $X=1$.   In (\ref{E:remainder}), we take $r_D=0$ for $D\subseteq \Lambda$.  With $\Zz$ the power set of $\Lambda$, we have $H=\prod_{\lambda\in \Lambda}(1+\omega_\lambda/(1-\omega_\lambda))=\prod_{\lambda\in \Lambda}(1-\omega_\lambda)^{-1}$ and hence our sieve gives the optimal bound $\mu(S)\leq \prod_{\lambda\in \Lambda} (1-\omega_\lambda)$.

In the general setting, we think of the sets $A_\lambda$ as being ``almost independent'' and hence the number $r_D$ should be relatively small (at least for some $D$ of small cardinality).  Inclusion-exclusion gives
\[
\mu(S)= \sum_{D\subseteq\Lambda} (-1)^{|D|} \mu(A_D) = \sum_{D\subseteq\Lambda} (-1)^{|D|} \Big(\prod_{\lambda\in D} \omega_\lambda \Big) X  + R = \prod_{\lambda\in \Lambda} (1-\omega_\lambda)\cdot X + R
\]
with $R:=\sum_{D\subseteq\Lambda} (-1)^{|D|} r_D$.     In practice, the ``error term'' $R$ can be difficult to control and may in fact exceed the ``main term''.  To find upper bounds for $\mu(S)$ using our sieve, one need to prudently select the \defi{sieve support} $\Zz$ so that ``error term'' in (\ref{E:main bound}) is not too large.

\subsection{Proof of Theorem \ref{T:Selberg sieve}}
For $D\subseteq \Lambda$, define $\omega_D=\prod_{\lambda\in D} \omega_\lambda$.  For each non-empty $D\in\Zz$, we fix a real number $\lambda_D$ that will be chosen later. Set $\lambda_\emptyset=1$.  For any $U\subseteq A$, let $\chi_U\colon A\to\{0,1\}$ be the characteristic function of $U$, i.e., $\chi_U(a)=1$ if and only if $a\in U$.  The set $U$ is measurable if and only if $\chi_U\colon A\to \{0,1\}$ is measurable.  For each $a\in A$, we claim that
\[
\chi_S(a) \leq \Big(\sum_{D\in\Zz}   \chi_{A_D}(a) \lambda_D \Big)^2.
\]
If $a\notin S$, then $\chi_S(a)=0$ and the above inequality is immediate since the square of a real number is non-negative.  If $a\in S$, then $\sum_{D\in\Zz}  \chi_{A_D}(a) \lambda_D = \lambda_\emptyset =1$.  Therefore,
\begin{align*} 
\mu(S) &= \int_A \chi_S(a) d\mu(a) \leq   \int_A \Bigl( \sum_{D\in \Zz}  \chi_{A_D}(a) \lambda_D \Bigr)^2 d\mu(a)= \sum_{D,D'\in\Zz} \Bigl( \int_A \chi_{A_D}(a) \chi_{A_{D'}}(a) d\mu(a) \Bigr) \lambda_D \lambda_{D'}
\end{align*}
and thus $\mu(S)\leq \sum_{D,D'\in\Zz} \mu(A_{D\cup D'}) \lambda_D \lambda_{D'}$. Using (\ref{E:remainder}), this inequality becomes
\[
\mu(S) \leq \Delta \cdot X+ R
\]
where 
\[
\Delta= \sum_{D,D'\in\Zz} \omega_{D\cup D'} \lambda_D\lambda_{D'} \text{\quad and \quad} R=\sum_{D,D'\in\Zz} r_{D\cup D'} \lambda_D \lambda_{D'}.
\]

We first study $\Delta$.  By the multiplicative definition of $\omega_D$, we have
\[
\Delta = \sum_{D,D'\in\Zz} \frac{\omega_D\omega_{D'}}{\omega_{D\cap D'}} \lambda_D \lambda_{D'}.
\]
For $D,D'\in \Zz$, we have
\[
\frac{1}{\omega_{D\cap D'}}=\prod_{\lambda\in D\cap D'} \Big(1 +  \frac{1-\omega_\lambda}{\omega_\lambda} \Big) = \sum_{E \subseteq D\cap D'} \prod_{\lambda\in E} \frac{1-\omega_\lambda}{\omega_\lambda}
\]
and thus
\begin{align*} 
\Delta &= \sum_{D,D'\in\Zz} \omega_D \omega_{D'} \Bigl(\sum_{E \subseteq D\cap D'} \prod_{\lambda\in E} \frac{1-\omega_\lambda}{\omega_\lambda} \Bigr) \lambda_D \lambda_{D'} = \sum_{E\in\Zz} \Bigl(\prod_{\lambda\in E} \frac{1-\omega_\lambda}{\omega_\lambda} \Bigr) \sum_{\substack{D,D' \in \Zz\\E\subseteq D, E\subseteq D'}} \omega_D \omega_{D'} \lambda_D \lambda_{D'}. 
\end{align*}
So
\begin{equation} \label{E:quadratic}
\Delta = \sum_{E\in\Zz} \Bigl(\prod_{\lambda\in E} \frac{1-\omega_\lambda}{\omega_\lambda}\Bigr) \xi_E^2
\end{equation}
where $\displaystyle\xi_E := (-1)^{|E|} \sum_{E\subseteq D \in \Zz} \omega_D \lambda_D$ for $E\in \Zz$.   By M\"obius inversion, for $D\in \Zz$ we have
\begin{equation} \label{E:lambda relation}
\omega_D \lambda_D = \sum_{D\subseteq E \in \Zz} (-1)^{|E|-|D|} \cdot (-1)^{|E|}\xi_E = (-1)^{|D|} \sum_{D\subseteq E \in \Zz} \xi_E
\end{equation}
and in particular, $\sum_{E\in \Zz} \xi_E = \lambda_\emptyset=1$.

Since $\Delta$ shows up in our upper bound for $\mu(S)$, we now minimize its value.   With (\ref{E:quadratic}) we view $\Delta$ as a quadratic form in the variables $(\xi_E)_{E\in\Zz}$ subject to the constraint $\sum_{E\in \Zz} \xi_E =1$; it is not hard to show that $\Delta$ obtains its minimum value of $H^{-1}=\big(\sum_{D\in\Zz} \prod_{\lambda\in D} \frac{\omega_\lambda}{1-\omega_\lambda}\big)^{-1}$ when
\[
\xi_E = \frac{1}{H}\prod_{\lambda \in E} \frac{\omega_\lambda}{1-\omega_\lambda} 
\]
for $E\in\Zz$.  With these optimized values of $\xi_E$ and (\ref{E:lambda relation}), we now \emph{define} 
\begin{equation} \label{E:lambda definition}
\lambda_D := \frac{1}{H}\frac{(-1)^{|D|}}{\omega_D} \sum_{ D\subseteq E\in\Zz } \prod_{\lambda \in E} \frac{\omega_\lambda}{1-\omega_\lambda} 
\end{equation}
for each $D\in\Zz$.
By our choice, we have $\Delta = H^{-1}$ and hence $\mu(S)\leq X/H +R$.   It remains to bound $R$.  For each $D\in \Zz$,
\begin{align*} 
0 \leq (-1)^{|D|} \lambda_D & = \frac{1}{H}\prod_{\lambda\in D}\Bigl(1 + \frac{\omega_\lambda}{1-\omega_\lambda}\Bigr) \sum_{ D\subseteq E\in\Zz } \prod_{\lambda \in E-D} \frac{\omega_\lambda}{1-\omega_\lambda} \leq  \frac{1}{H} \sum_{ E\in\Zz } \prod_{\lambda \in E} \frac{\omega_\lambda}{1-\omega_\lambda} = 1.
\end{align*}
Therefore, 
\[
R \leq \sum_{D,D'\in \Zz} |r_{D\cup D'}| |\lambda_D||\lambda_{D'}| \leq \sum_{D,D'\in \Zz} |r_{D\cup D'}|.\qedhere
\]

\section{Equidistribution} \label{S:B}

Let $U$ be an affine variety of dimension $d\geq 1$ over a finite field $\FF_q$ that is geometrically smooth and irreducible.   Let $\rho\colon \pi_1(U) \to G$ be a surjective and continuous homomorphism, where $\pi_1(U)$ is the \'etale fundamental group and $G$ is a finite group.   Let $G^g$ be the image of $\pi_1(U_{\Fbar_q})$ under $\rho$ and define $m=[G:G^g]$.  We have an exact sequence of groups
\[
1\to G^g \hookrightarrow G \xrightarrow{\varphi} \ZZ/m\ZZ \to 1
\]
such that $\varphi(\Frob_u)\equiv n \pmod{m}$ for all $u\in U(\FF_{q^n})$.

\begin{theorem}
\label{T:equidistribution}
Fix an integer $n\geq 1$.   Let $C$ be a subset of $G$ that is stable under conjugation and satisfies $\varphi(C)=\{n \bmod{m}\}$.  
\begin{romanenum}
\item  \label{T:equidistribution i}
Then
\[
\frac{|\{u \in U(\F_{q^n}): \rho(\Frob_u)\subseteq C \}|}{|U(\F_{q^n})|} =  \frac{|C|}{|G^g|} + O(q^{-n/2}),
\]
where the implicit constant does not depend on $n$.

\item \label{T:equidistribution ii}
Assume further that $U$ is of dimension 1 and $\rho$ is tamely ramified.  Let $X/\FF_q$ be the smooth projective curve obtained by completing $U$.   Let $g$ be the genus of $X$ and define $b= |X(\FFbar_q)- U(\FFbar_q)|$.   Suppose that $H\subseteq G^g$ is a normal subgroup of $G$ that satisfies $C\cdot H\subseteq C$.  Then
\[
\Bigl| |\{u \in U(\F_{q}): \rho(\Frob_u)\subseteq C \}| - \frac{|C|}{|G^g|}|U(\FF_q)| \Bigr| \leq \frac{|C|^{1/2}}{|H|^{1/2}} (1-|H|/|G^g|)^{1/2} (2g-2 + b) q^{1/2}.
\]
\end{romanenum}
\end{theorem}

\begin{proof}
Both parts are applications of the machinery of Grothendieck and Deligne used to prove the Weil conjectures.  Part (i) is well known; a proof can be found in \S4 of \cite{MR1440067}.  For (ii), one can replace $\rho$ with the representation $\pi_1(U,\eta)\xrightarrow{\rho} G \to G/H$ and reduce to the case where $H=1$.  This case has already been dealt with by the author, cf.~\cite{1011.6465}*{Proposition~5.1}.
\end{proof}



\begin{bibdiv}
\begin{biblist}

\bib{Atlas}{book}{
      author={Conway, J.~H.},
      author={Curtis, R.~T.},
      author={Norton, S.~P.},
      author={Parker, R.~A.},
      author={Wilson, R.~A.},
       title={Atlas of finite groups},
   publisher={Oxford University Press},
     address={Eynsham},
        date={1985},
        ISBN={0-19-853199-0},
        note={Maximal subgroups and ordinary characters for simple groups, With
  computational assistance from J. G. Thackray},
      label={ATLAS},
}

\bib{MR2381481}{article}{
      author={Ahmadi, Omran},
      author={Vega, Gerardo},
       title={On the parity of the number of irreducible factors of
  self-reciprocal polynomials over finite fields},
        date={2008},
        ISSN={1071-5797},
     journal={Finite Fields Appl.},
      volume={14},
      number={1},
       pages={124\ndash 131},
}



\bib{MR2588789}{article}{
   author={Beauville, Arnaud},
   title={The primitive cohomology lattice of a complete intersection},
   language={English, with English and French summaries},
   journal={C. R. Math. Acad. Sci. Paris},
   volume={347},
   date={2009},
   number={23-24},
   pages={1399--1402},
   issn={1631-073X},
}


\bib{MR1440067}{article}{
      author={Chavdarov, Nick},
       title={The generic irreducibility of the numerator of the zeta function
  in a family of curves with large monodromy},
        date={1997},
        ISSN={0012-7094},
     journal={Duke Math. J.},
      volume={87},
      number={1},
       pages={151\ndash 180},
}

\bib{MR2200366}{book}{
      author={Cojocaru, Alina~Carmen},
      author={Murty, M.~Ram},
       title={An introduction to sieve methods and their applications},
      series={London Mathematical Society Student Texts},
   publisher={Cambridge University Press},
     address={Cambridge},
        date={2006},
      volume={66},
}


\bib{MR2183392}{article}{
   author={Conrad, B.},
   author={Conrad, K.},
   author={Helfgott, H.},
   title={Root numbers and ranks in positive characteristic},
   journal={Adv. Math.},
   volume={198},
   date={2005},
   number={2},
   pages={684--731},
   issn={0001-8708},
}

\bib{MR0332694}{incollection}{
      author={Gallagher, P.~X.},
       title={The large sieve and probabilistic {G}alois theory},
        date={1973},
   booktitle={Analytic number theory ({P}roc. {S}ympos. {P}ure {M}ath., {V}ol.
  {XXIV}, {S}t. {L}ouis {U}niv., {S}t. {L}ouis, {M}o., 1972)},
   publisher={Amer. Math. Soc.},
     address={Providence, R.I.},
       pages={91\ndash 101},
}

\bib{MR1947324}{article}{
      author={Gross, Benedict~H.},
      author={McMullen, Curtis~T.},
       title={Automorphisms of even unimodular lattices and unramified {S}alem
  numbers},
        date={2002},
        ISSN={0021-8693},
     journal={J. Algebra},
      volume={257},
      number={2},
       pages={265\ndash 290},
}

\bib{MR2372151}{article}{
      author={Hall, Chris},
       title={Big symplectic or orthogonal monodromy modulo {$\ell$}},
        date={2008},
        ISSN={0012-7094},
     journal={Duke Math. J.},
      volume={141},
      number={1},
       pages={179\ndash 203},
}

\bib{MR2061214}{book}{
      author={Iwaniec, Henryk},
      author={Kowalski, Emmanuel},
       title={Analytic number theory},
      series={American Mathematical Society Colloquium Publications},
   publisher={American Mathematical Society},
     address={Providence, RI},
        date={2004},
      volume={53},
}

\bib{MR2539184}{article}{
      author={Jouve, Florent},
       title={Maximal {G}alois group of {$L$}-functions of elliptic curves},
        date={2009},
        ISSN={1073-7928},
     journal={Int. Math. Res. Not. IMRN},
      number={19},
       pages={3557\ndash 3594},
}

\bib{MR1081536}{book}{
   author={Katz, Nicholas M.},
   title={Exponential sums and differential equations},
   series={Annals of Mathematics Studies},
   volume={124},
   publisher={Princeton University Press, Princeton, NJ},
   date={1990},
   pages={xii+430},
   isbn={0-691-08598-6},
   isbn={0-691-08599-4},
}

\bib{katzreport}{incollection}{
      author={Katz, Nicholas~M.},
       title={Report on the irreducibility of ${L}$-functions},
        date={2012},
   booktitle={Number theory, analysis and geometry},
      editor={Goldfeld, Dorian},
      editor={Jorgenson, Jay},
      editor={Jones, Peter},
      editor={Ramakrishnan, Dinakar},
      editor={Ribet, Kenneth A.~A.},
      editor={Tate, John},
   publisher={Springer New York},
       pages={321\ndash 353},
}



\bib{MR1659828}{book}{
   author={Katz, Nicholas M.},
   author={Sarnak, Peter},
   title={Random matrices, Frobenius eigenvalues, and monodromy},
   series={American Mathematical Society Colloquium Publications},
   volume={45},
   publisher={American Mathematical Society, Providence, RI},
   date={1999},
   pages={xii+419},
   isbn={0-8218-1017-0},
}

\bib{MR1370110}{article}{
   author={Larsen, M.},
   title={Maximality of Galois actions for compatible systems},
   journal={Duke Math. J.},
   volume={80},
   date={1995},
   number={3},
   pages={601--630},
   issn={0012-7094},
}

\bib{MR0022871}{article}{
      author={Selberg, Atle},
       title={On an elementary method in the theory of primes},
        date={1947},
     journal={Norske Vid. Selsk. Forh., Trondhjem},
      volume={19},
      number={18},
       pages={64\ndash 67},
}

\bib{MR2017446}{book}{
       title={Rev\^etements \'etales et groupe fondamental ({SGA} 1)},
      series={Documents Math\'ematiques (Paris) [Mathematical Documents
  (Paris)], 3},
   publisher={Soci\'et\'e Math\'ematique de France},
     address={Paris},
        date={2003},
        ISBN={2-85629-141-4},
        note={S{\'e}minaire de g{\'e}om{\'e}trie alg{\'e}brique du Bois Marie
  1960--61. [Algebraic Geometry Seminar of Bois Marie 1960-61], Directed by A.
  Grothendieck, With two papers by M. Raynaud, Updated and annotated reprint of
  the 1971 original [Lecture Notes in Math., 224, Springer, Berlin; MR0354651
  (50 \#7129)]},
label={SGA1},      
}

\bib{MR0354657}{book}{
   title={Groupes de monodromie en g\'eom\'etrie alg\'ebrique. II},
   language={French},
   series={Lecture Notes in Mathematics, Vol. 340},
   note={S\'eminaire de G\'eom\'etrie Alg\'ebrique du Bois-Marie 1967--1969 (SGA 7
   II);
   Dirig\'e par P. Deligne et N. Katz},
   publisher={Springer-Verlag, Berlin-New York},
   date={1973},
   pages={x+438},
   label={SGA7 II},
}

\bib{MR1967689}{book}{
   author={Voisin, Claire},
   title={Hodge theory and complex algebraic geometry. I},
   series={Cambridge Studies in Advanced Mathematics},
   volume={76},
   note={Translated from the French original by Leila Schneps},
   publisher={Cambridge University Press, Cambridge},
   date={2002},
   pages={x+322},
   isbn={0-521-80260-1},
}

\bib{MR2562037}{book}{
      author={Wilson, Robert~A.},
       title={The finite simple groups},
      series={Graduate Texts in Mathematics},
   publisher={Springer-Verlag London Ltd.},
     address={London},
        date={2009},
      volume={251},
        ISBN={978-1-84800-987-5},
}

\bib{MR0148760}{article}{
      author={Zassenhaus, Hans},
       title={On the spinor norm},
        date={1962},
        ISSN={0003-9268},
     journal={Arch. Math.},
      volume={13},
       pages={434\ndash 451},
}

\bib{Zywina-Maximal}{article}{
      author={Zywina, David},
       title={Elliptic curves with maximal {G}alois action on their torsion
  points},
        date={2010},
     journal={Bull. London Math. Soc.},
      volume={42},
      number={5},
       pages={811\ndash 826},
}

\bib{1011.6465}{article}{
      author={Zywina, David},
       title={Hilbert's irreducibility theorem and the larger sieve},
        date={2010},
      eprint={https://arxiv.org/abs/1011.6465},
        note={arXiv:1011.6465},
}

\bib{orthogonal}{article}{
      author={Zywina, David},
       title={The inverse Galois problem for orthogonal groups},
        date={2014},
      eprint={https://arxiv.org/abs/1409.1151},
        note={arXiv:1409.1151},
}

\end{biblist}
\end{bibdiv}
\end{document}